\documentclass[11pt]{amsart}
\usepackage{amssymb}
\usepackage{latexsym} 
\usepackage{color}

\def\nc{\newcommand}

\newcommand{\les}{\lesssim}

\nc\pa{\partial}

\def\longequals{\mathbin{=\kern-2pt=}}

\nc\CC{\mathbb{C}}
\nc\RR{\mathbb{R}}
\nc\QQ{\mathbb{Q}}
\nc\ZZ{\mathbb{Z}}
\nc\NN{\mathbb{N}}

\newcommand{\wt}{\widetilde}

\nc\m[1]{\left| #1\right|}
\nc\norm[1]{\left\| #1\right\|}
\def\longequals{\mathbin{=\kern-2pt=}}

\newtheorem{theorem}{Theorem}[section]
\newtheorem{lemma}[theorem]{Lemma}
\newtheorem{corollary}[theorem]{Corollary}
\newtheorem{proposition}[theorem]{Proposition}
\newtheorem{remark}[theorem]{Remark}        
\numberwithin{equation}{section}

\begin{document}

\title[Stationary Navier-Stokes equations]
{Stationary Navier-Stokes equations with critically singular external forces: 
existence and stability results}

\author[Tuoc Van Phan]
{Tuoc Van Phan }
\address{Department of Mathematics,
University of Tennessee,
277 Ayres Hall, 1403 Circle Drive,
Knoxville, TN 37996.}
\email{phan@math.utk.edu}

\author[Nguyen Cong Phuc]
{Nguyen Cong Phuc$^{*}$}
\address{Department of Mathematics,
Louisiana State University,
303 Lockett Hall, Baton Rouge, LA 70803, USA.}
\email{pcnguyen@math.lsu.edu}

\thanks{$^*$Supported in part by NSF Grant DMS-0901083}

\begin{abstract} We show the unique existence of solutions to stationary  Navier-Stokes equations with small 
singular external forces belonging to a critical space. 
To the best of our knowledge, this is the largest critical space that is available up to now for this kind of existence. 
This result can be viewed  as the stationary counterpart of the existence result obtained by H. Koch and D. Tataru 
for the free non-stationary Navier-Stokes equations
with initial data in $\textup{BMO}^{-1}$. 
The stability of the stationary solutions 
in such spaces is also obtained by a series  of sharp  estimates for  resolvents of a singularly perturbed operator and the 
corresponding semigroup.
\end{abstract}

\maketitle

\section{Introduction}\label{Introduction}

In this paper we address the existence and stability problem for the forced stationary 
Navier-Stokes equations describing the motion of incompressible fluid in the whole space $\RR^n$, $n\geq 3$:
\begin{eqnarray}\label{NV}
\left\{\begin{array}{rcl}
U\cdot\nabla U +\nabla P &=& \Delta U + F,\\
\nabla\cdot U&=&0.
\end{array}
\right.
\end{eqnarray}
Here $U=(U_1, \dots, U_n): \mathbb{R}^n \mapsto \mathbb{R}^n$ is an unknown velocity of the fluid, 
$P:\mathbb{R}^n \mapsto \mathbb{R}$ is an unknown pressure, and $F=(F_1,\dots, F_n): \mathbb{R}^n \mapsto \mathbb{R}^n$ 
is a given external force potentially with strong singularities. 

To put our results in perspective, let us first discuss related results 
concerning the Cauchy problem for the free non-stationary Navier-Stokes equations  with possibly
 irregular initial data:
\begin{eqnarray}\label{Cauchy}
\left\{\begin{array}{rcll}
u_t+ u\cdot\nabla u +\nabla p &=& \Delta u, \qquad & {\rm in~}  \RR^n \times (0,\infty),\\
\nabla\cdot u&=&0, \qquad & {\rm in~}  \RR^n \times (0,\infty),\\
u(0, \cdot)&=& u_0, \qquad & {\rm in~}  \RR^n.
\end{array}
\right.
\end{eqnarray}

In 1984, T. Kato \cite{Ka1} initiated the study of \eqref{Cauchy} with initial data belonging to the 
space $L^n(\RR^n)$ and obtained global existence in a subspace 
of $C([0,\infty), L^n)$ provided the norm $\norm{u_0}_{L^n}$ is small enough. This kind of global 
existence with small initial data continues to hold also for homogeneous Morrey spaces 
$\mathcal{M}^{p, \,p}(\RR^n)$, $1\leq p\leq n$; see \cite{Ka2}, and also \cite{Tay, KY2}. Here for $1\leq p<\infty$
 and $0<\lambda\leq n$,  we say that a function 
$f\in L^{p}_{\rm loc}(\RR^n)$ belongs to the Morrey space $\mathcal{M}^{p,\, \lambda}(\RR^n)$ provided 
that its norm
$$\norm{f}_{\mathcal{M}^{p,\, \lambda}(\RR^n)}= \sup_{B_r(x_0)\subset\RR^n} \left\{ r^{\lambda -n}\int_{B_r(x_0)} |f(x)|^p dx
\right\}^{\frac{1}{p}}<+\infty.$$
When $p=1$ one allows  $f$ to be   a locally finite measure in $\RR^n$ in which case the $L^1$ norm should
 be replaced by the total variation. Note that our notation of Morrey spaces in this paper
is different from those in, e.g., \cite{Ka2, Tay, KY1, KY2}.

Later in 2001, H. Koch and D. Tataru \cite{KT} showed that global well-posedness of the Cauchy problem holds as well for 
small initial data in the space $\text{BMO}^{-1}$. This space  can be defined as the space of all distributional divergences of 
$\text{BMO}$ 
vector fields.   It is well-known that the following continuous embeddings hold
\begin{equation}\label{spacesembed}
L^n\subset \mathcal{M}^{p, \,p}\subset \text{BMO}^{-1}\subset B^{-1}_{\infty,\, \infty},
\end{equation}
where the last one is a homogeneous Besov space consisting of distributions $f$ for which the norm 
$$\norm{f}_{B^{-1}_{\infty,\, \infty}}=\sup_{t>0} t^{\frac{1}{2}} \norm{e^{t\Delta}f(\cdot)}_{L^{\infty}}<+\infty.$$

On the other hand, it has been shown  recently by J. Bourgain and N. Pavlov\'{i}c \cite{BP} that the Cauchy problem 
with initial data in $B^{-1}_{\infty,\, \infty}$ is ill-posed no matter how small the initial data are. See also  \cite{M-S}
for an earlier result on a Navier-Stokes like equation, and the recent paper \cite{Yo} for ill-posedness in a space even 
smaller than $B^{-1}_{\infty,\, \infty}$.

All of the spaces appear in \eqref{spacesembed} are invariant with respect to the scaling 
$f(\cdot)\rightarrow \lambda f(\lambda \cdot), \lambda>0,$ in the sense that $\norm{f}_{E}=\norm{\lambda f(\lambda \cdot)}_{E}$ 
for all $\lambda>0$.  
Thus up to now $\text{BMO}^{-1}$ is the largest space invariant under such a scaling on which the 
Cauchy problem is well-posed for small initial data.

However, the situation is much more subtle when it comes to  forced stationary Navier-Stokes equations. 
It is well-known that  system \eqref{NV} is invariant under the scaling $(U, P, F)\mapsto (U_\lambda, P_\lambda, F_\lambda)$, 
where $U_\lambda= \lambda U(\lambda \cdot)$, $P_\lambda= \lambda^2 P(\lambda \cdot)$ and $F_\lambda= \lambda^3 F(\lambda \cdot)$ 
for all $\lambda>0$. Moreover, it can be recast into an integral equation
\begin{equation*}
U=\Delta^{-1} \mathbb{P} \nabla\cdot (U \otimes U) -\Delta^{-1} \mathbb{P} F,
\end{equation*}
where $\mathbb{P}= \text{Id}-\nabla \Delta^{-1} \nabla \cdot$ is the Helmholtz-Leray projection onto the vector fields of
 zero divergence. Thus
in order to give a meaning to the nonlinear term in \eqref{NV} one wants the solution $U$ to
be at least in $L^{2}_{\rm loc}(\RR^n)$. On the other hand, the largest Banach space $X\subset L^{2}_{\rm loc}(\RR^n)$
 that is invariant 
under translation and that $\norm{U_\lambda}_{X}=\norm{U}_X$ is the Morrey space $\mathcal{M}^{2,\, 2}$ (see \cite{Me}).
 Thus one is tempted to 
look for solutions in $\mathcal{M}^{2,\, 2}$ under the smallness condition 
$$\norm{(-\Delta)^{-1} F}_{\mathcal{M}^{2,\, 2}}\leq \delta.$$

However, as noted in \cite{BBIS}, it seems impossible to prove such existence results under this condition
 as for $U\in \mathcal{M}^{2,\, 2}$
the matrix $U\otimes U$  would belong to $\mathcal{M}^{1,\, 2}$, but unfortunately the first order Riesz potentials of functions in
 $\mathcal{M}^{1,\, 2}$ may not even belong to $L^{2}_{\rm loc}(\RR^n)$. It is worth mentioning that existence
 and uniqueness hold (see \cite{KY2}) 
 under a  ``bump" condition:
\begin{equation} \label{bumpcond}
\norm{(-\Delta)^{-1} F}_{\mathcal{M}^{2+\epsilon,\, 2+\epsilon}}\leq \delta
\end{equation}
for some $\epsilon>0$, in which case the solution $U$ belongs to  $\mathcal{M}^{2+\epsilon,\, 2+\epsilon}$.

We observe  that condition \eqref{bumpcond} is not sharp and propose in this paper the existence in a
 larger space called $\mathcal{V}^{1,\,2}(\RR^n)$
under the smallness condition 
\begin{equation*}
\norm{(-\Delta)^{-1} F}_{\mathcal{V}^{1,\, 2}}\leq \delta.
\end{equation*}
Here $\mathcal{V}^{1,\,2}(\RR^n)$ is the space of all locally square integrable functions $f$ for which there is 
a constant $C_{f}\geq 0$ such that the inequality 
$$\left(\int_{\RR^n} |\varphi|^2 |f|^2 dx\right)^{\frac{1}{2}} \leq C_{f} \left(\int_{\RR^n}|\nabla \varphi|^2
 dx\right)^{\frac{1}{2}}$$
holds for all $\varphi\in C^{\infty}_{0}(\RR^n)$.  The space $\mathcal{V}^{1,\,2}(\RR^n)$
 is well-known and it can be characterized as 
$$\mathcal{V}^{1,\, 2}(\RR^n)= \{ f\in L_{\rm loc}^{2}(\RR^n): \norm{f}_{\mathcal{V}^{1,\, 2}(\RR^n)}<+\infty \},$$
where the norm is defined by
$$\norm{f}_{\mathcal{V}^{1,\, 2}(\RR^n)}= \sup \left[\frac{\int_{K} |f|^2 dx}{{\rm cap}_{1, \, 2}(K)}\right]^{\frac{1}{2}}$$
with the supremum being taken over all compact sets $K\subset\RR^n$ of positive capacity ${\rm cap}_{1, \, 2}(K)$
 (see Section \ref{sec2}). Here the 
capacity ${\rm cap}_{1, \, 2}(\cdot)$ is defined for each compact set $K$ by
\begin{equation}\label{cap12sob}
{\rm cap}_{1,\, 2}(K)=\inf\Big\{\int_{\RR^n}|\nabla \varphi|^2  dx: \varphi\in C^\infty_0(\RR^n),
\varphi\geq \chi_K \Big\}.
\end{equation}

This space has been used for example in \cite{MV, HMV} for Riccati type equations and in \cite{PG, L-R, LR, L-RM}
under the notation $\dot{X}_1$  to obtain a sharp 
weak-strong uniqueness result for the Navier-Stokes equations. 
It is worth mentioning that capacities and spaces similar to
$\mathcal{V}^{1,\, 2}$ have also been used as indispensable tools to treat 
various Lane-Emden type and Riccati type equations in \cite{AP, PhV, Ph}.

Note that we have the following continuous embedding
$$\mathcal{M}^{2+\epsilon,\, 2+\epsilon}(\RR^n)\subset\mathcal{V}^{1,\, 2}(\RR^n)\subset \mathcal{M}^{2,\, 2}(\RR^n)$$ 
for any $\epsilon>0$. The second inclusion is easy to see from the fact that
\begin{equation*}
 {\rm cap}_{1, \, 2}(B_r(x)) \simeq r^{n-2},
\end{equation*}
whereas the first one was obtained by C. Fefferman and D. H. Phong 
in the analysis  of Schr\"odinger  operators (see \cite{Fef}).  It is now known that the first inclusion can also be improved 
further by replacing the Morrey space on the left-hand side by other larger ones of Dini or Orlicz type (see \cite{ChWW, P}).

On the other hand, we remark that the inclusion $\mathcal{V}^{1,\, 2}(\RR^n)\subset \mathcal{M}^{2,\, 2}(\RR^n)$ is strict as 
by \cite[Proposition 3.6]{MV} there is an $f\in \mathcal{M}^{2,\, 2}(\RR^n)$ given in terms of the first order Riesz potential of a compactly supported 
measure such that $f$ fails to be in $\mathcal{V}^{1, \, 2}$. In other words, one  cannot take $\epsilon=0$ in the first inclusion. 

It is easy to see from the definition of ${\rm cap}_{1, \, 2}$ that it is translation invariant and that 
$${\rm cap}_{1, \, 2}(K)= \lambda^{2-n}{\rm cap}_{1, \, 2}(\lambda K) $$
for all $\lambda>0$ and compact sets $K\subset\RR^n$.
Thus the space  $\mathcal{V}^{1,\, 2}\subset \mathcal{M}^{2,\, 2}$ is also translation invariant and satisfies 
$\norm{U_\lambda}_{\mathcal{V}^{1,\, 2}}=\norm{U}_{\mathcal{V}^{1,\, 2}}$. Therefore, it seems that  $\mathcal{V}^{1,\, 2}$
 the best candidate for 
this problem as it includes all Morrey spaces of the form $\mathcal{M}^{2+\epsilon,\, 2+\epsilon}$ for any $\epsilon>0$.

Throughout of this paper, the notation $A \les B$ means that there is a universal constant $C>0$ such that $A \leq CB$. Also, 
for a vector function $f: \mathbb{R}^n \mapsto \mathbb{R}^n$, we say that $f \in \mathcal{V}^{1,\, 2}$ if and only if 
$|f| \in \mathcal{V}^{1,\, 2}$ and $\norm{f}_{\mathcal{V}^{1,\, 2}} = \norm{|f|}_{\mathcal{V}^{1,\, 2}}$. We are now ready to state
 the first result of the paper.
\begin{theorem} \label{Stationary-Sol} There exists a 
sufficiently small number $\delta_0 >0$ such that if 
$\norm{(-\Delta)^{-1}F}_{\mathcal{V}^{1,\, 2}} <\delta_0$, then the system of equations \eqref{NV} 
has unique solution $U$ satisfying
\[ \norm{U}_{\mathcal{V}^{1,\, 2}} \les  \norm{(-\Delta)^{-1}F}_{\mathcal{V}^{1,\, 2}}.\]
\end{theorem}

To the best of our knowledge, up to now $\mathcal{V}^{1,\, 2}$ is the largest space that is invariant under 
translation and the scaling $f(\cdot)
\rightarrow \lambda f(\lambda\cdot)$ on which existence holds for the stationary Navier-Stokes equations with small external forces. 
See also Remark \ref{V12opt} below for further discussion about the optimality of $\mathcal{V}^{1,\, 2}$. 
Thus Theorem 
\ref{Stationary-Sol} can be viewed as the stationary counterpart of the result obtained by H. Koch and D. Tataru 
for the Cauchy problem with small 
initial data in $\text{BMO}^{-1}$ discussed earlier.

\begin{remark} Let $v=|(-\Delta)^{-1}F|^2$ and $\varphi:[0, \infty)\rightarrow [1, \infty)$ be an increasing function such that
$$\int_{1}^{\infty}\frac{1}{t \varphi(t)} dt <+\infty.$$
Suppose that 
\begin{equation}\label{Dinicon}
\sup_{B_r\subset\RR^n} r^{2-n}\int_{B_r} v(x)\varphi(v(x)r^2) dx <+\infty.
\end{equation}
Then by the result in \cite{ChWW} one has $(-\Delta)^{-1}F$ belongs to $\mathcal{V}^{1,\, 2}$. On the other hand, by taking for example 
$$\varphi(t)= \log(e+t)^{1+\delta} \quad or \quad \varphi(t)=\log(e+t) [\log\log(e+t)]^{1+\delta}$$
for some $\delta>0$, we see that condition  \eqref{Dinicon} is strictly weaker than the condition $(-\Delta)^{-1}F\in \mathcal{M}^{2+\epsilon, \, 2+\epsilon}$,
$\epsilon>0$, used in \cite{KY2}. 
\end{remark}

The second result of this paper is about the stability of   solutions to  stationary Navier-Stokes equations in the 
space $\mathcal{V}^{1,\, 2}$. To
this end,  we consider the non-stationary Navier-Stokes equations
\begin{eqnarray}\label{NSE}
\left\{\begin{array}{rcll}
u_t+ u\cdot\nabla u +\nabla p &=& \Delta u + F\qquad & {\rm in~}  \RR^n \times (0,\infty),\\
\nabla\cdot u&=&0 \qquad & {\rm in~}  \RR^n \times (0,\infty),\\
u(0, \cdot)&=& u_0 \qquad & {\rm in~}  \RR^n,
\end{array}
\right.
\end{eqnarray}
where $F$ is as in Theorem \ref{Stationary-Sol}  which is time-independent, and $u_0$ is  a divergence-free vector 
field in $\mathcal{V}^{1,\, 2}$. 
Our stability results say that if the difference $\norm{u_0 -U}_{\mathcal{V}^{1,\, 2}}$, where $U$ is as obtained in Theorem \ref{Stationary-Sol}, 
is sufficiently small, then there exists a unique global in time solution $u$ of \eqref{NSE}. Moreover, as time $t$ goes to infinity
the non-stationary solution $u$  of \eqref{NSE} converges to $U$  in a suitable space. 
  
\begin{theorem} \label{stability-th} Let $\sigma_0 \in (1/2,1)$. 
There exist two positive numbers $\epsilon_0$ and $\delta_1$ with $\delta_1 \leq \delta_0$ 
such that for  $\norm{(-\Delta)^{-1}F}_{\mathcal{V}^{1,\, 2}} <\delta_1$, the following existence and uniqueness results hold.
For every $u_0$ satisfying $\nabla \cdot u_0 =0$ and $\norm{u_0 - U}_{\mathcal{V}^{1,\, 2}} < \epsilon_0$, 
there exists uniquely a time-global solution $u(x,t)$ of \eqref{NSE} satisfying
\begin{equation}\label{V1/2}
\sup_{t >0} t^{1/4}\Vert (-\Delta)^{1/4}(u(\cdot,t) - U)\Vert_{\mathcal{V}^{1,\, 2}} \leq C  \norm{u_0 - U}_{\mathcal{V}^{1,\, 2}},
\end{equation}
with the initial condition understood in the sense that  
\[ \sup_{t >0}t^{\alpha/2}\Vert (-\Delta)^{\alpha/2}(u(\cdot,t) - u_0)\Vert_{\mathcal{V}^{1,\, 2}}
 \leq C \norm{u_0 - U}_{\mathcal{V}^{1,\, 2}} \]
for every $\alpha \in [-1, 0]$. Moreover, for every $\sigma \in [0, \sigma_0]$, the solution $u$ enjoys the time-decay estimate
\begin{equation} \label{sigma-u.est}
 \sup_{t >0}t^{\sigma/2}\Vert (-\Delta)^{\sigma/2}(u(\cdot,t) - U)\Vert_{\mathcal{V}^{1,\, 2}} 
\leq C(\sigma_0)  \norm{u_0 - U}_{\mathcal{V}^{1,\, 2}}.
\end{equation}
\end{theorem}

In Theorem \ref{stability-th} the notation $(-\Delta)^{s/2}$, $s\in\RR$, stands for a non-local  fractional derivative of order $s$ 
whose precise definition will be given in the next section.  
We now have the following remarks on Theorem \ref{stability-th}.

\begin{remark} 
\begin{itemize}
\item[\textup{(i)}] At each time $t>0$, both of the solutions $u$ and $U$ may not be in the same space. 
\item[\textup{(ii)}] When $\sigma =0$, the estimate \eqref{sigma-u.est} provides the 
Lyapunov stability of the stationary solution $U$. 
Moreover, it also implies that the solution $u$ remains in $\mathcal{V}^{1,\, 2}$ at all time.
\item[\textup{(iii)}] A similar stability result in the Morrey 
space $\mathcal{M}^{2+\epsilon,\, 2+\epsilon}$ with $\epsilon>0$ can be found
\cite{KY2}.
\end{itemize}
\end{remark}

We would like to emphasize that one can easily prove the existence of the solution $u$ satisfying 
only \eqref{sigma-u.est} with $\sigma =0$ by following the approach of Calder{\'{o}}n \cite{CC, CCP} 
and Cannone \cite{MC} (see Remark \ref{MC-remark} below for more details). This approach, 
however, only yields the Lyapunov stability of $U$. 
On the other hand, we observe that the approach to prove the stability of $U$ in \cite{BBIS} 
is not enough for our purpose. 
To prove Theorem \ref{stability-th} we instead
follow a semi-group approach in a  spirit similar to that of \cite{KY2}. This approach, though complicated, 
gives us both Lyapunov and asymptotics stability of $U$ at the same time, as well as 
the regularity of solutions in the fractional Sobolev spaces associated to 
$\mathcal{V}^{1,\, 2}$.    
Due to possible strong singularities carried along by stationary solutions, 
new sharp and delicate  decay estimates must be 
obtained  for  resolvents of a singularly perturbed operator and the corresponding semigroup. 
We achieve those by tactically combining spectral theory methods
with  some hard analysis in potential theory and harmonic analysis such as capacitary inequalities and 
weighted norm inequalities for singular integrals 
and multiplier operators.

In this paper, besides the main purpose of studying  stationary Navier-Stokes 
equations, we also give a careful analysis on Sobolev spaces 
associated to $\mathcal{V}^{1,\, 2}$ such as  Sobolev type embedding theorems and interpolations 
between them, some of which have been developed in \cite{MS1, MS2}. We expect that such an analysis, as well as the complicated study of the resolvent of the perturbed Laplacian, will  be useful for other purposes as well.

The organization of the paper is as follows. Section \ref{sec2} is devoted to some preliminaries on potential theory and 
harmonic analysis to understand the space $\mathcal{V}^{1,\, 2}(\RR^n)$ and homogeneous Sobolev spaces associated to it.
In Section \ref{sec3} we show the existence and uniqueness of stationary solutions in the space $\mathcal{V}^{1, \, 2}(\RR^n)$.
Sharp decay estimates for resolvents and analytic semigroups generated by singularly  perturbed operators 
are carried out in Section \ref{sec4}. Finally, the analysis in Section \ref{sec4} is applied in Section \ref{sec5} to obtain
the stability of stationary solutions in $\mathcal{V}^{1,\, 2}(\RR^n)$.
 
\section{Preliminaries}\label{sec2}

For a  non-negative locally finite measure $\mu$ in $\RR^n$, $n\geq 3$, the  Riesz potential of order $\gamma\in (0, n)$ of $\mu$ 
is defined  by
$${\rm\bf I}_{\gamma}*\mu(x)= c(n,\gamma)\int_{\RR^n} \frac{d\mu(y)}{|x-y|^{n-\gamma}}, \qquad x\in\RR^n,$$
where 
$$c(n, \gamma)=\Gamma(\frac{1}{2}(n-\gamma))/2^\gamma \pi^{n/2}\Gamma(\frac{1}{2}\gamma)$$ 
is a normalizing constant. It is known that (see, e.g., \cite{La}) a necessary and sufficient condition 
for the finiteness almost everywhere of ${\rm\bf I}_{\gamma}*\mu$ is the inequality
$$\int_{|y|>1} \frac{d\mu(y)}{|y|^{n-\gamma}}<+\infty,$$
in which case $\mu$ belongs to  the space of tempered distributions $\mathcal{S}'(\RR^n)$.

For every compact set   $K\subset\RR^n$ it is known that one has the following equivalence 
\begin{equation*}
{\rm cap}_{1, \, 2}(K) \simeq \inf\{\norm{f}_{L^2(\RR^n)}^{2}: f\geq 0, {\rm\bf I}_{1}*f\geq
1 {\rm ~on~} K\},
\end{equation*}
where the capacity ${\rm cap}_{1,\, 2}(K)$ is as defined in \eqref{cap12sob}.

We begin with the following special case of  Theorem 2.1 in \cite{MV}.

\begin{theorem}\label{CHAR} Let  $\nu$ be  a non-negative locally finite
measure in $\RR^n$, $n\geq 3$. Then  the following properties of $\nu$
are equivalent.

{\rm (i)} There is a constant $A_1>0$ such that
\begin{equation*}
\int_{\RR^n} u^2 d\nu \leq A_1 \int_{\RR^n} |\nabla u|^{2}dx
\end{equation*}
for all  $u\in C_0^{\infty}(\RR^n)$.

{\rm (ii)} There is a constant $A_2>0$ such that
\begin{equation*}
\int_{\RR^n} ({\rm\bf I}_{1}*f)^{2}d\nu \leq A_2 \int_{\RR^n} f^{2}dx
\end{equation*}
for all non-negative $f\in L^{2}(\RR^n)$.

{\rm (iii)} There is a constant $A_3>0$ such that
\begin{equation} \label{Mazmeasure}
\nu(K)\leq A_3\, {\rm cap}_{1, \, 2}(K)
\end{equation}
for all compact sets $K \subset \RR^n$.

{\rm (iv)} There is a constant $A_4>0$ such that

\begin{equation*}
\int_{K} ({\rm\bf I}_{1}*\nu)^{2}dx   \leq A^{2}_4\, {\rm cap}_{1, \, 2}(K)
\end{equation*}
for all compact sets $K\subset\RR^n$.

Moreover, the least possible values of the constants $A_1$, $A_2$, $A_3$, and $A_4$ are  comparable to each other.
\end{theorem}

We denote by $\mathfrak{M}_{+}^{1,\, 2}$ the class of non-negative measure $\nu$ 
for which \eqref{Mazmeasure} holds
for all compact sets $K\subset\RR^n$ with norm
$$\norm{\nu}_{\mathfrak{M}_{+}^{1,\, 2}}= \sup\left\{ \frac{\nu(K)}{{\rm cap}_{1, \, 2}(K)}: 
K ~{\rm compact}~\subset\RR^n, {\rm cap}_{1, \, 2}(K)>0\right\}.$$

For our purpose we also need the following spaces. For $\alpha\in \RR$ we define
\begin{equation}\label{spacedef}
\mathcal{V}^{1,\, 2}_{\alpha}= \{f\in \mathcal{S}'/\mathcal{P}: 
\, \text{with norm}\, \norm{f}_{\mathcal{V}^{1,\, 2}_{\alpha}}= \norm{(-\Delta)^{\frac{\alpha}{2}} f}_{\mathcal{V}^{1,\, 2}}<+\infty\},
\end{equation}
where $\mathcal{P}$ is the set of all polynomials in $\RR^n$, and  $(-\Delta)^{\frac{\alpha}{2}} f$ is defined for each $f\in \mathcal{S}'/\mathcal{P}$ by 
$$(-\Delta)^{\frac{\alpha}{2}} f= \mathcal{F}^{-1}(|\xi|^{\alpha} \mathcal{F}(f)(\xi))\in \mathcal{S}'/\mathcal{P}.$$

We remark that $\mathcal{S}'/\mathcal{P}$ can be identified with $\mathcal{S}_{\infty}'$ where $\mathcal{S}_{\infty}$ is a 
closed  subspace of 
$\mathcal{S}$ characterized by
$$\varphi\in \mathcal{S}_{\infty} \Longleftrightarrow  \langle P, \varphi\rangle =0, \qquad \forall\, P\in \mathcal{P}.$$

Therefore, explicitly we have 
$$\langle (-\Delta)^{\frac{\alpha}{2}} f, \varphi\rangle = \langle f, \mathcal{F}(|\xi|^{\alpha} \mathcal{F}^{-1}(\varphi)) \rangle \qquad \forall \,
\varphi\in \mathcal{S}_{\infty}.$$
This is well defined as one can check, for any $\alpha\in \RR$, that $\mathcal{F}(|\xi|^{\alpha} \mathcal{F}^{-1}(\varphi))$ 
belongs to $\mathcal{S}_{\infty}$ whenever
$\varphi$ belongs to $\mathcal{S}_{\infty}$.

Since zero is the only polynomial that belongs to ${\mathcal{V}^{1,\, 2}}$ we see that ${\mathcal{V}^{1,\, 2}}$  injects 
in $\mathcal{S}'/\mathcal{P}$. As a consequence the condition 
$$\norm{(-\Delta)^{\frac{\alpha}{2}} f}_{\mathcal{V}^{1,\, 2}}<+\infty$$
in \eqref{spacedef} simply means that there is a unique $h\in {\mathcal{V}^{1,\, 2}}$ that belongs to the equivalence class  $(-\Delta)^{\frac{\alpha}{2}} f$,
and thus we have $\norm{f}_{\mathcal{V}^{1,\, 2}_{\alpha}}=\norm{h}_{\mathcal{V}^{1,\, 2}}$.

It follows from Theorem \ref{CHAR} that a non-negative measure $\nu$ belongs to  $\mathfrak{M}^{1,\, 2}_{+}$ if and only if it belongs to $\mathcal{V}^{1,\, 2}_{-1}$ and moreover,
\begin{equation}\label{MVequi}
c_1\norm{\nu}_{\mathfrak{M}_{+}^{1,\, 2}}\leq \norm{\nu}_{\mathcal{V}^{1,\, 2}_{-1}}\leq c_2\norm{\nu}_{\mathfrak{M}_{+}^{1,\, 2}}
\end{equation}
for some constants $c_1$ and $c_2$ depending only on $n$.

As $\mathcal{V}^{1,\, 2}\subset \mathcal{M}^{2,\, 2}$ we see that $\mathcal{V}^{1,\, 2}_{\alpha}\subset \mathcal{M}^{2,\, 2}_{\alpha}$, where the space 
$\mathcal{M}^{2,\, 2}_{\alpha}$ is defined in a 
similar way based on $\mathcal{M}^{2,\, 2}$. Thus for $\alpha<1$ every equivalent class of  $\mathcal{V}^{1,\, 2}_{\alpha}$ has a canonical representative in $\mathcal{S}'(\RR^n)$;
see \cite{KY1, KY2, Bo}.

Our approach in this paper is based on the following boundedness property of 
Riesz transforms on the space $\mathcal{V}_{\alpha}^{1,\, 2}(\RR^n)$,
whose proof was already given in \cite{MV}.

\begin{theorem}\label{riesz} For any $j=1, \dots, n$ and $\alpha\in\RR$ one has 
$$\norm{R_j f}_{\mathcal{V}_{\alpha}^{1,\, 2}(\RR^n)}\leq C \norm{f}_{\mathcal{V}_{\alpha}^{1,\, 2}(\RR^n)},$$
where $R_j$ is the $j$-th Riesz transform defined by $R_j f=\partial_j (-\Delta)^{-\frac{1}{2}} f.$ 
\end{theorem}

\begin{corollary}\label{pro} Let $\mathbb{P}= Id-\nabla \Delta^{-1} \nabla \cdot$ be the 
Helmholtz-Leray projection onto the divergence-free vector fields. Then one has the following bound:
$$\norm{\mathbb{P} f}_{\mathcal{V}^{1,\, 2}_{\alpha}}\leq C \norm{f}_{\mathcal{V}^{1,\, 2}_{\alpha}}$$
for all $\alpha\in \RR$. 
\end{corollary}

More generally, we have the following mapping property of singular integrals on the space $\mathcal{V}^{1,\, 2}_{\alpha}$.

\begin{theorem}\label{singular} Let $\sigma, s\in \RR$ and let $P(\xi)$ be a $C^{n}$-function on $\RR^n\setminus\{0\}$ that satisfies
\begin{equation}\label{HP}
\left|\frac{\partial^{|\alpha|} P}{\partial \xi^\alpha}(\xi)\right| \leq C |\xi|^{\sigma -|\alpha|}
\end{equation}
for all multi-indices $\alpha\in \NN^n$ with $|\alpha|\leq n$ and all $\xi\in \RR^n\setminus\{0\}$. Then the Fourier multiplier
 operator $P(D)$ 
is bounded from $\mathcal{V}^{1,\, 2}_{s}$ to $\mathcal{V}^{1,\, 2}_{s-\sigma}$.
\end{theorem}

\begin{proof} For given $f\in \mathcal{V}^{1,\, 2}_{s}$, let $g=(-\Delta)^{s/2}f \in \mathcal{V}^{1,\, 2}$. We need to show that 
\begin{equation}\label{DPg}
\norm{(-\Delta)^{-\sigma/2} P(D) g}_{\mathcal{V}^{1,\, 2}}\leq C \norm{g}_{\mathcal{V}^{1,\, 2}}. 
\end{equation}
The symbol of the operator  $(-\Delta)^{-\sigma/2} P(D)$ is given by $m(\xi)=|\xi|^{-\sigma} P(\xi)$ for $\xi\in \RR^n\setminus\{0\}$.
Thus by \eqref{HP} we see that $m(\xi)$ satisfies the following Mikhlin's condition
\begin{equation*}
\left|\frac{\partial^{|\alpha|} m}{\partial \xi^\alpha}(\xi)\right| \leq C |\xi|^{-|\alpha|}.
\end{equation*}
for all multi-indices $\alpha\in \NN^n$ with $|\alpha|\leq n$ and all $\xi\in \RR^n\setminus\{0\}$.
Then it follows from \cite[Theorem 2]{KW} that the following weighted estimate 
\begin{equation}\label{KW}
\int_{\RR^n}|(-\Delta)^{-\sigma/2} P(D)g (x)|^p w(x)dx  \leq C \int_{\RR^n}|g (x)|^p w(x)dx
\end{equation}
holds for all $1<p<\infty$ provided the weight $w$ belongs the Muckenhoupt class $A_p$. In particular, \eqref{KW} 
holds if $w$ belongs the the class
$A_1$, i.e., if $w$ satisfies the pointwise bound
\begin{equation}\label{A1cond}
{\rm\bf M} w(x) \leq A w(x)
\end{equation}  
for a.e. $x\in \RR^n$ and for a fixed constant $A\geq 1$. In \eqref{A1cond}, {\rm \bf M} stands for  
the Hardy-Littlewood maximal function defined for each $f\in L^{1}_{\rm loc}(\RR^n)$ by
$${\rm\bf M}f(x)= \sup_{r>0} \frac{1}{|B_r(x)|}\int_{B_r(x)} |f(y)| dy.$$

Finally, applying Lemma 3.1 in \cite{MV} we obtain  the bound \eqref{DPg}.
\end{proof}

In regard to  the the Hardy-Littlewood maximal function {\rm\bf M},
 we have  the following useful boundedness result which is also  a consequence of \cite[Lemma 3.1]{MV}.
\begin{theorem}\label{IV} Let $1<p<\infty$ and $n\geq 3$. Then
$$\sup \frac{\int_{K} |{\rm\bf M}f|^p dx}{{\rm cap}_{1,\,2}(K)}\les \sup \frac{\int_{K} |f|^p dx}{{\rm cap}_{1,\,2}(K)}, $$
where the suprema are taken over all compact sets $K\subset\RR^n$ with positive ${\rm cap}_{1,\,2}(K)$.
\end{theorem}

\section{Stationary Navier-Stokes Equations}\label{sec3}

The goal of this section is to prove Theorem \ref{Stationary-Sol}. We first recall the following standard fixed point 
lemma which is useful in solving Navier-Stokes equations with small data (see, e.g., \cite{Me}).
\begin{lemma} \label{fixedpoint}
Let $X$ be a Banach space with norm $\norm{\cdot}_{X}$ and let $B: X\times X\rightarrow X$ be a bi-linear map such that 
$$\norm{B(x,y)}_{X}\leq \alpha \norm{x}_{X}\norm{y}_{X}$$
for all $x, y\in X$ and some $\alpha>0$. Then for each $y_0\in X$ with $4\alpha \norm{y_0}_{X}<1$ the equation 
$$x=B(x,x)+ y_0$$
has a  solution $x\in X$, and moreover this is the only solution for which 
$$\norm{x}_{X}\leq 2\norm{y_0}_{X}.$$
\end{lemma}

\begin{proof}[Proof of Theorem \ref{Stationary-Sol}]
We will  apply Lemma \ref{fixedpoint} to our context by choosing $X=\mathcal{V}^{1,\, 2}(\RR^n)$ and letting
$$U_0=-\Delta^{-1}\mathbb{P} F,\qquad B(U,V)=\Delta^{-1}\mathbb{P} \nabla\cdot (U\otimes V),$$
where $$F=(F_1, \dots, F_n)\in \mathcal{V}^{1,\, 2}_{-2},$$ 
and $$U=(U_1, \dots, U_n), V=(V_1, \dots, V_n) \in \mathcal{V}^{1,\, 2}.$$

Then by \eqref{MVequi} and  H\"older's inequality we have 
$$\norm{(U\otimes V)}_{\mathcal{V}^{1,\, 2}_{-1}} \leq C 
\norm{(U\otimes V)}_{\mathfrak{M}_{+}^{1,\, 2}}\leq C \norm{U}_{\mathcal{V}^{1,\, 2}}\norm{V}_{\mathcal{V}^{1,\, 2}}.$$

Thus it follows from Corollary \ref{pro} that 
$$B:\mathcal{V}^{1,\, 2}\times \mathcal{V}^{1,\, 2}\rightarrow \mathcal{V}^{1,\, 2}$$
with
$$\norm{B(U, V)}_{\mathcal{V}^{1,\, 2}}\leq C_1 \norm{U}_{\mathcal{V}^{1,\, 2}}\norm{V}_{\mathcal{V}^{1,\, 2}}.$$

On the other hand, by  Corollary \ref{pro} we have 
$$\norm{U_0}_{\mathcal{V}^{1,\, 2}}\leq C_2\delta_0$$

Finally choosing small $\delta_0>0$ so that $4\delta_0 C_1 C_2 <1 $ and
 applying Lemma \ref{fixedpoint} we obtain a unique solution to \eqref{NV}. 
\end{proof}

\begin{remark}\label{V12opt} Here we further discuss the question of optimality of $\mathcal{V}^{1,\, 2}$. For $u\in L_{\rm loc}^{2}(\RR^n)$ define inductively 
$w_0= |u|$ and $w_{n+1}= {\rm\bf I}_{1}(w_n^2)$ for $n\geq 0$. Let $X$ be the space  
$$X=\{u\in L_{\rm loc}^{2}(\RR^n): \norm{u}_{X}<+\infty\},\quad \text{with} \ \norm{u}_{X} =\sup_{n\geq 0}\norm{w_n}_{\mathcal{M}^{2,\, 2}}^{\frac{1}{2^n}}.$$

Then it is easy to see that the bi-linear map $(u, v) \mapsto {\rm\bf I}_{1}*(uv)$ is
bounded on $X\times X$ and that $X$ satisfies the property
$\norm{v}_{X}\leq \norm{u}_{X}$ whenever $u\in X$ and $|v|\leq u$ a.e.

Thus the  equation \eqref{NV} can be solved in $X$ for any external force $F$ with small $\norm{(-\Delta)^{-1}F}_{X}$. 
Moreover, any space $Y$ such that $Y\subset \mathcal{M}^{2,\,2}$ and 
$(u, v) \mapsto {\rm\bf I}_{1}*(uv)$ is bounded on $Y\times Y$ satisfies
$Y\subset X$. In particular, this gives $\mathcal{V}^{1,\, 2}\subset X$. As a matter of  fact, we have 
\begin{equation}\label{IdenX}
X =\mathcal{V}^{1,\, 2}.
\end{equation}

The identification \eqref{IdenX} has been known, even in a much more general setting, 
in the beautiful work \cite{KV} \textup{(}see Theorems 2.10, 5.6, and 5.7 in \cite{KV}\textup{)}. 
In our context,  the
inclusion $X \subset \mathcal{V}^{1,\, 2}$ can be shown as follows. Given $f\in X$, 
by Lemma \ref{fixedpoint}  the equation
$$u= {\rm\bf I}_{1}*(u^2) +\epsilon\, |f| \qquad a.e.$$
has a non-negative solution $u\in X$ for some $\epsilon>0$. Next, by  Lemma 4.1 in \cite{VW} for any non-negative $g\in L^2$ there holds
$$({\rm\bf I}_{1}*g)^2 \leq c(n)\,  {\rm\bf I}_{1}*(g \, {\rm\bf I}_{1}*g),$$
and thus
\begin{eqnarray*}
\int_{\RR^n}({\rm\bf I}_{1}*g)^2 u^2 dx &\leq& c \int_{\RR^n} [{\rm\bf I}_{1}*(g \, {\rm\bf I}_{1}*g)] u^2 dx \\
&=& c \int_{\RR^n} (g \, {\rm\bf I}_{1}*g) [{\rm\bf I}_{1}*(u^2)] dx\\
&\leq & c \norm{g}_{L^2} \left(\int_{\RR^n} ({\rm\bf I}_{1}*g)^2 [{\rm\bf I}_{1}*(u^2)]^2 dx\right)^{\frac{1}{2}}\\
&\leq & c \norm{g}_{L^2} \left(\int_{\RR^n} ({\rm\bf I}_{1}*g)^2 u^2 dx\right)^{\frac{1}{2}},
\end{eqnarray*}
where we used that $u^2\geq [{\rm\bf I}_{1}*(u^2)]^2$ in the last inequality.
This yields 
$$\int_{\RR^n}({\rm\bf I}_{1}*g)^2 u^2 dx \leq C \int_{\RR^n} g^2 dx,$$
and as $u^2\geq \epsilon^2 |f|^2$ we find
$$\int_{\RR^n}({\rm\bf I}_{1}*g)^2 |f|^2 dx \leq C \epsilon^{-2} \int_{\RR^n} g^2 dx,$$ 
which holds for all non-negative $g\in L^2$.
Thus by Theorem \ref{CHAR} we see that $f\in \mathcal{V}^{1,\, 2}$, i.e., the inclusion $X \subset \mathcal{V}^{1,\, 2}$ holds true. 

\end{remark}

\section{Resolvent Estimates and Analytic  Semigroups} \label{sec4}

In this section we consider the non-stationary Navier-Stokes equations \eqref{NSE}
where $F$ is as in Theorem \ref{Stationary-Sol}  and $u_0$ is assumed to be in $\mathcal{V}^{1,\, 2}$ with
zero divergence. 

Recall that in Theorem \ref{stability-th} we want  to show that if $\norm{u_0 -U}_{\mathcal{V}^{1,\, 2}}$ is sufficiently small, 
then there exists a unique time-global solution $u$ of \eqref{NSE}. Moreover, 
the difference $u -U$ will converge to zero in a suitable space 
as time $t\rightarrow \infty$.   Here, $U$ is the stationary solution 
of \eqref{NV} whose existence is guaranteed by Theorem \ref{Stationary-Sol}. Let us define
\begin{equation}\label{KZ.def}
\begin{split}
\mathcal{B}[f](x) & = \mathbb{P} \nabla \cdot [U(\cdot) \otimes f(\cdot) + f(\cdot) \otimes U(\cdot)](x), \\
\mathcal{A}[f](x) & = -\Delta f(x) + \mathcal{B}[f](x).
\end{split}
\end{equation}
Then, for $w = u-U $ and $w^0 =  u_0-U$, the system \eqref{NSE} can be written as
\begin{equation}\label{KZ}
\left \{ \begin{aligned}
& \frac{\partial w}{\partial t}(\cdot,t) + \mathcal{A}[w(\cdot, t] + \mathbb{P}\nabla \cdot [w(\cdot,t) \otimes w(\cdot,t)] =0\\
& w(0, \cdot) = w^0(\cdot).
\end{aligned} \right.
\end{equation}
Therefore, using Duhamel's principle, one has the integral form of \eqref{KZ} 
\begin{equation} \label{inte.form} w(\cdot,t) = e^{-\mathcal{A} t}w^0 - 
\int_0^t e^{-\mathcal{A}(t-s)}\mathbb{P}\nabla\cdot[w(\cdot, s)\otimes w(\cdot, s)]ds. 
\end{equation}

We shall show in the next section that \eqref{KZ} and \eqref{inte.form} are equivalent and use \eqref{inte.form} to 
prove the existence and uniqueness of solution $w$ which converges to zero in some suitable space. 
To this end, we need to make sense of the semigroup $e^{-\mathcal{A}t}$ and characterize its properties. 
That will be the main objective of this section.

We first recall a pointwise estimate of Riesz potentials due to D. R. Adams \cite{Ad} that will be needed shortly.

\begin{lemma}\label{AD} Let $0<\alpha<\beta \leq n/p$, $p\in (1, \infty)$. Then for any $f\in \mathcal{M}^{p, \, \beta p}(\RR^n)$ we have
$${\rm \bf I}_{\alpha}*f(x) \les \norm{f}_{\mathcal{M}^{p, \, \beta p}}^{\alpha/\beta} ({\rm \bf M}f(x))^{(\beta-\alpha)/\beta}.$$
\end{lemma}

The following Sobolev type embedding theorem will be essential to our development later. 
The idea behind its proof is due to Igor E. Verbitsky (see \cite{MS1, MS2}).

\begin{theorem}\label{Sobcap} Let $1<p<\infty$ and suppose that $f$ is a function that satisfies
$$\sup \frac{\int_{K}|f|^p dx}{{\rm cap}_{1,\, 2}(K)}<+\infty.$$
Then for any $0<\alpha<2/p$ we have 
$$\sup \left[\frac{\int_{K}|{\rm\bf I}_{\alpha}*f|^{\frac{2p}{2-\alpha p}}}{{\rm cap}_{1,\, 2}(K)} \right]^{\frac{2-\alpha p}{2p}}\les \sup \left[\frac{\int_{K}|f|^p dx}{{\rm cap}_{1,\, 2}(K)}\right]^{\frac{1}{p}}. $$
\end{theorem}

\begin{proof} For simplicity we set 
$$A=\sup \left[\frac{\int_{K}|f|^p dx}{{\rm cap}_{1,\, 2}(K)}\right]^{\frac{1}{p}}.$$
 
Then by applying Lemma \ref{AD} with $\beta=2/p$ we find 
\begin{equation*}
{\rm \bf I}_{\alpha}*f(x) \les \norm{f}_{\mathcal{M}^{p, \, \beta p}}^{\alpha p/2} ({\rm \bf M}f(x))^{1-\alpha p/2} 
\les  A^{\alpha p/2} ({\rm \bf M}f(x))^{1-\alpha p/2}.
\end{equation*} 

Thus it follows from Theorem \ref{IV} that 
\begin{eqnarray*}
\sup \frac{\int_{K}|{\rm\bf I}_{\alpha}*f|^{\frac{2p}{2-\alpha p}}}{{\rm cap}_{1,\, 2}(K)}&\les& A^{\frac{\alpha p^2}{2-\alpha p}} \sup \frac{\int_{K}({\rm \bf M}f)^{p} dx}{{\rm cap}_{1,\, 2}(K)}\\
&\les& A^{\frac{\alpha p^2}{2-\alpha p}} A^{p}=A^{\frac{2 p}{2-\alpha p}},
\end{eqnarray*}
which yields the desired result.
\end{proof}

\begin{lemma}\label{Adomain} For any $s\in (0,1)$, there exists a constant $C =C(s) >0$ such that 
the operators $\mathcal{A}$ and  $\mathcal{B}$ are bounded from $\mathcal{V}^{1,\, 2}_{s}$ 
to $\mathcal{V}^{1,\, 2}_{s-2}$ with bounds 
$$\norm{\mathcal{B}}_{\mathcal{V}^{1,\, 2}_{s}\rightarrow \mathcal{V}^{1,\, 2}_{s-2}}\leq C\norm{U}_{\mathcal{V}^{1,\, 2}},$$
and
$$\norm{\mathcal{A}}_{\mathcal{V}^{1,\, 2}_{s}\rightarrow \mathcal{V}^{1,\, 2}_{s-2}}\leq C(1+\norm{U}_{\mathcal{V}^{1,\, 2}}).$$
\end{lemma}

\begin{proof} It is enough to show the conclusion for $\mathcal{B}$. To this end, let $f\in \mathcal{V}^{1,\, 2}_{s}$ and write $g= 
(-\Delta)^{s/2}f\in \mathcal{V}^{1,\, 2}$. Then $f= (-\Delta)^{-s/2}g$ and 
by Theorem \ref{Sobcap} with $p=2$ and $\alpha=s$ we have
\begin{equation} \label{So-in}
\int_{K}|f(x)|^{\frac{2}{1-s}} dx \les \norm{g}_{\mathcal{V}^{1,\, 2}}^{\frac{2}{1-s}}
 {\rm cap}_{1,\, 2}(K)=\norm{f}_{\mathcal{V}_{s}^{1,\, 2}}
^{\frac{2}{1-s}}{\rm cap}_{1,\, 2}(K)
\end{equation}
for all compact sets $K\subset\RR^n$. Thus by H\"older's inequality we get for each compact set $K$,
\begin{eqnarray}\label{SobfU} 
\left(\int_{K}|f\otimes U|^\frac{2}{2-s} dx\right)^{\frac{2-s}{2}} &\leq& 
\left(\int_{K} |f|^{\frac{2}{1-s}} dx\right)^{\frac{1-s}{2}} 
\left(\int_{K}|U|^2 dx\right)^{\frac{1}{2}}\\
&\les& \norm{f}_{\mathcal{V}_{s}^{1, \, 2}} \norm{U}_{\mathcal{V}^{1, \, 2}}{\rm cap}_{1,\,2}(K)^{\frac{2-s}{2}}.\nonumber
\end{eqnarray}

On the other hand,  by Theorem \ref{singular} we have  
\begin{equation*}
\begin{split}
&\Vert (-\Delta)^{(s-2)/2}\mathbb{P}\nabla \cdot [U \otimes f + f \otimes U] \Vert_{\mathcal{V}^{1,\, 2}}       \\
&= \Vert (-\Delta)^{(s-2)/2}\mathbb{P}\nabla\cdot (-\Delta)^{(1-s)/2}(-\Delta)^{(s-1)/2} [U \otimes f + 
f \otimes U] \Vert_{\mathcal{V}^{1,\, 2}}       \\
&\les \Vert (-\Delta)^{(s-2)/2}\mathbb{P}\nabla\cdot (-\Delta)^{(1-s)/2} \Vert_{\mathcal{V}_{s-1}^{1,\, 2}\rightarrow 
\mathcal{V}_{s-1}^{1,\, 2}}
\Vert U \otimes f + f \otimes U\Vert_{\mathcal{V}^{1,\, 2}_{s-1}}\\
&\les \Vert U \otimes f + f \otimes U\Vert_{\mathcal{V}^{1,\, 2}_{s-1}}.
\end{split}
\end{equation*}

To estimate the last term we use \eqref{SobfU} and apply Theorem \ref{Sobcap} with  $p=2/(2-s)$, $\alpha=1-s$ to get
\begin{eqnarray*}
\left[\sup \frac{\int_{K} ({\rm\bf I}_{1-s}*|f\otimes U|)^{2} dx}{{\rm cap}_{1,\,2}(K)}\right]^{\frac{1}{2}}
&\les& \left[\sup \frac{\int_{K} |f\otimes U|^\frac{2}{2-s} dx}{{\rm cap}_{1,\,2}(K)}\right]^{\frac{2-s}{2}}\nonumber\\
&\les& \norm{f}_{\mathcal{V}_{s}^{1, \, 2}} \norm{U}_{\mathcal{V}^{1, \, 2}}.
\end{eqnarray*}

Therefore, 
$$\Vert (-\Delta)^{(s-2)/2}\mathbb{P}\nabla \cdot [U \otimes f + f \otimes U] \Vert_{\mathcal{V}^{1,\, 2}} \les \norm{f}_{\mathcal{V}_{s}^{1, \, 2}} \norm{U}_{\mathcal{V}^{1, \, 2}} $$
and thus completes the proof.
\end{proof}

As an operator on $\mathcal{V}^{1,\, 2}_s$, $s\in \RR$, into itself, the domain of  $-\Delta$  is the sub-space $\mathcal{V}^{1,\, 2}_s \cap \mathcal{V}^{1,\, 2}_{s+2}$. 
By means of the Fourier transform on $\mathcal{S}'/\mathcal{P}$ and Theorem \ref{singular} we see that, for each $\lambda \in\mathbb{C}\setminus [0,\infty)$, the resolvent 
$(\lambda+\Delta)^{-1}$ is given by the Fourier multiplier operator
$$T_{\lambda}(f)=\mathcal{F}^{-1}[(\lambda-|\xi|^2)^{-1}\mathcal{F}(f)(\xi)].$$

More generally, we have the following bound on negative powers of $\lambda+\Delta$.
\begin{lemma} \label{Laplace.decay}  
Let $0 < \gamma <\pi/2$ and $S_{\gamma} = \{\lambda \in \mathbb{C} \setminus [0, \infty): 
|\textup{arg}(\lambda)| \geq \gamma \}$. Then for any $0 \leq a \leq b$, there exists a constant
$C =C(n,\gamma,a,b)$ such that for all $\lambda \in S_\gamma$
\[ \Vert (-\Delta)^{a}(\lambda +\Delta)^{-b}\Vert_{\mathcal{V}^{1,\, 2} \rightarrow \mathcal{V}^{1,\, 2}}
\leq C|\lambda|^{a-b}.  \]
\end{lemma}
\begin{proof} This lemma follows directly from Theorem \ref{singular}.
\end{proof}

Next, for each $s\in (0, 1)$, a domain of the operator $\mathcal{A}$ on $\mathcal{V}^{1,\, 2}_s$ is naturally given by
\[ D(\mathcal{A}) = \{ f \in \mathcal{V}^{1,\, 2}_s : \mathcal{A}[f] \in \mathcal{V}^{1,\, 2}_s \}. \]

Fix now $0 < \gamma <\pi/2$ and let $\lambda \in S_{\gamma}$, where $S_{\gamma}$ is as in Lemma \ref{Laplace.decay}. 
By Lemmas \ref{Adomain} and \ref{Laplace.decay} we see that $(\lambda +\Delta)^{-1}\mathcal{B}: 
\mathcal{V}^{1,\, 2}_s\rightarrow \mathcal{V}^{1,\, 2}_s$ with bound
$$\norm{(\lambda +\Delta)^{-1}\mathcal{B}}_{\mathcal{V}^{1,\, 2}_s\rightarrow \mathcal{V}^{1,\, 2}_s}\leq 
M \norm{U}_{\mathcal{V}^{1,\, 2}},
$$where $M$ depends only on $s$ and $\gamma$.
Thus when $\norm{U}_{\mathcal{V}^{1,\, 2}}< \frac{1}{2M}$, the operator $1-(\lambda +\Delta)^{-1}\mathcal{B}$ 
is invertible whose inverse is given by a Von Neumann series:
\begin{equation} \label{Von-1}
[1-(\lambda +\Delta)^{-1}\mathcal{B}]^{-1}=\sum_{j=0}^{\infty} [(\lambda+\Delta)^{-1}\mathcal{B}]^j \end{equation}
on $\mathcal{V}^{1,\, 2}_s$, with   
$$\norm{[1-(\lambda +\Delta)^{-1}\mathcal{B}]^{-1}}_{\mathcal{V}^{1,\, 2}_s\rightarrow \mathcal{V}^{1,\, 2}_s}\leq 
\frac{1}{1-M\norm{U}_{\mathcal{V}^{1,\, 2}}} \leq 2.$$

It is then easy to check that, for such $\lambda$ and $U$, the operator $\lambda-\mathcal{A}$ is invertible with 
\begin{equation} \label{Von-2}
(\lambda-\mathcal{A})^{-1}=[1-(\lambda +\Delta)^{-1}\mathcal{B}]^{-1}(\lambda +\Delta)^{-1} \end{equation}
on $\mathcal{V}^{1,\, 2}_s$. Moreover, it follows from Lemma \ref{Laplace.decay} and the commutativity of 
$(\lambda +\Delta)^{-1}$ and $(-\Delta)^{\frac{-s}{2}}$
that $(\lambda-\mathcal{A})^{-1}$ is bounded on $\mathcal{V}^{1,\, 2}_s$ with
\begin{equation} \label{R.s.est}
\norm{(\lambda-\mathcal{A})^{-1}}_{\mathcal{V}^{1,\, 2}_s\rightarrow \mathcal{V}^{1,\, 2}_s}\leq C|\lambda|^{-1}, \quad
\ \forall\,  \lambda \in S_\gamma. 
\end{equation}

This shows that when $\norm{U}_{\mathcal{V}^{1,\, 2}}$ is sufficiently small  the sector 
$S_{\gamma}$ is contained in the resolvent set of $\mathcal{A}$.
The following lemma says even stronger that, 
in fact, $(\lambda-\mathcal{A})^{-1}$ maps boundedly from $\mathcal{V}^{1,\, 2}_s$ into $\mathcal{V}^{1,\, 2}_\sigma$ 
for all $s, \sigma\in (-2,1)$ such that
$s\leq\sigma\leq2+s$. 

In what follows, $\gamma$ is a fixed number in $(0, \pi/2)$ and $S_{\gamma}$ is as defined in Lemma \ref{Laplace.decay}.

\begin{lemma} \label{Res} Let $\alpha, \sigma$ be in $(-2,1)$,  $|\sigma - \alpha| \leq 2$. 
Then, there exists $\epsilon_1 =\epsilon_1(\alpha, \sigma)$ 
such that if $\norm{U}_{\mathcal{V}^{1,\, 2}} < \epsilon_1$, the operator 
$(\lambda+\Delta)^{-1}\mathcal{B}(\lambda-\mathcal{A})^{-1}$  can be extended to a bounded map from 
$\mathcal{V}^{1,\, 2}_{\alpha}$ to $\mathcal{V}^{1,\, 2}_{\sigma}$ for all $\lambda \in S_\gamma$ and 
the extension enjoys the estimate
\begin{equation} \label{se.est}
\Vert (\lambda+\Delta)^{-1}\mathcal{B}(\lambda-\mathcal{A})^{-1}\Vert_{
\mathcal{V}^{1,\, 2}_\alpha\rightarrow \mathcal{V}^{1,\, 2}_\sigma} \leq C(\alpha, \sigma)|\lambda|^{(\sigma-\alpha)/2-1}, \ 
\quad \forall\,  \lambda\, \in S_\gamma.
\end{equation}
Moreover, if $\alpha \leq \sigma$, the operator $(\lambda-\mathcal{A})^{-1}$ can 
be extended to a bounded operator 
from $\mathcal{V}^{1,\, 2}_{\alpha}$ to $\mathcal{V}^{1,\, 2}_{\sigma}$ for all $\lambda \in S_\gamma$ with
\begin{equation} \label{R.est} \Vert (\lambda-\mathcal{A})^{-1} \Vert_{
\mathcal{V}^{1,\, 2}_{\alpha}\rightarrow \mathcal{V}^{1,\, 2}_{\sigma}} \leq C(\alpha, \sigma) |\lambda|^{(\sigma-\alpha)/2-1}, 
\quad  \forall\, \lambda \in S_\gamma.
 \end{equation}
\end{lemma}

\begin{proof} We follow the approach in \cite{KY2} using Lemmas \ref{Adomain}--\ref{Laplace.decay}. 
For each $\lambda \in S_\gamma$, we set $\mathcal{R}(\lambda)=(\lambda-\mathcal{A})^{-1}$ and write
\begin{equation} \label{A.expension}
\begin{split}
& (-\Delta)^{\sigma/2}\mathcal{R}(\lambda)(-\Delta)^{-\alpha/2} \\
& = (-\Delta)^{\sigma/2}(\lambda +\Delta)^{-1}(-\Delta)^{-\alpha/2} \\ 
& \qquad\qquad +(-\Delta)^{\sigma/2}(\lambda+\Delta)^{-1}\mathcal{B}\mathcal{R}(\lambda)(-\Delta)^{-\alpha/2}\\
& = (-\Delta)^{(\sigma-\alpha)/2}(\lambda +\Delta)^{-1}
+ (-\Delta)^{\sigma/2}(\lambda+\Delta)^{-1}\mathcal{B}\mathcal{R}(\lambda)(-\Delta)^{-\alpha/2}.
\end{split}
\end{equation}

For $0 \leq \sigma -\alpha \leq 2$, it follows from Lemma \ref{Laplace.decay} that 
there is a constant $C_0 = C_0(\alpha, \sigma)$ 
such that the penultimate term in  \eqref{A.expension} can be controlled as
\begin{equation*} 
\Vert (-\Delta)^{(\sigma-\alpha)/2}(\lambda +\Delta)^{-1}\Vert_{\mathcal{V}^{1,\, 2} \rightarrow \mathcal{V}^{1,\, 2}}
\leq C_0|\lambda|^{(\sigma-\alpha)/2 -1}.  \end{equation*}

Therefore, to obtain \eqref{se.est} and \eqref{R.est}, it suffices to control the last term in the right hand side of 
\eqref{A.expension} for all 
 $\alpha, \sigma \in (-2, 1)$ with $|\sigma -\alpha| \leq 2$. Note that for such $\alpha, \sigma$ we have 
\begin{equation*}
 [\max\{\sigma, \alpha\}/2, 1+ \min\{\sigma, \alpha\}/2 ] \cap (0,1/2) \not= \phi,
\end{equation*}
and thus, we can find a real number $\theta\in (0, 1/2)$ such that
\begin{equation}\label{theta.cond}
 0 \leq \sigma/2 +1-\theta \leq 1, \quad   0\leq \theta-\alpha/2 \leq 1. 
\end{equation}

With this choice of $\theta$, we let $\mathcal{C}$ be an operator on $\mathcal{V}^{1,\, 2}$ defined by
$$\mathcal{C}[f]=(-\Delta)^{\theta-1}\mathcal{B}(-\Delta)^{-\theta}(f), \qquad\qquad f\in \mathcal{V}^{1,\, 2}.$$

Then, it follows from \eqref{theta.cond} and  Lemmas \ref{Adomain}--\ref{Laplace.decay} that there 
are constants $C_1$ and $C_2 = C_2(\alpha, \sigma)$ such that
\begin{equation}\label{C-ope}
\Vert (\lambda +\Delta)^{-1}(-\Delta)
 \Vert_{\mathcal{V}^{1,\, 2}\rightarrow \mathcal{V}^{1,\, 2}} \leq C_1, 
\quad \Vert \mathcal{C} \Vert_{\mathcal{V}^{1,\, 2}\rightarrow \mathcal{V}^{1,\, 2}} \leq C_2 \norm{U}_{\mathcal{V}^{1,\, 2}}.
\end{equation}
Hence, if $\norm{U}_{\mathcal{V}^{1,\, 2}} < \epsilon_1 = \frac{1}{2C_1C_2}$, the series 
$$\sum_{j=0}^\infty \left \{(\lambda +\Delta)^{-1}(-\Delta)\mathcal{C}\right\}^{j}$$
converges in the space $\mathcal{L}(\mathcal{V}^{1,\, 2}, \mathcal{V}^{1,\, 2})$ of all linear bounded operators 
from $\mathcal{V}^{1,\, 2}$ into $\mathcal{V}^{1,\, 2}$. Moreover,
\begin{equation}\label{series.ope}
\left \Vert \sum_{j=0}^\infty \left \{(\lambda +\Delta)^{-1}(-\Delta)\mathcal{C}\right\}^{j} \right
\Vert_{\mathcal{V}^{1,\, 2}\rightarrow \mathcal{V}^{1,\, 2}} \leq 2.
\end{equation}

On the other hand, it follows from \eqref{theta.cond} and  Lemmas \ref{Adomain}--\ref{Laplace.decay} that 
there is $C_3 = C_3(\alpha, \sigma)$ such that
\begin{equation}\label{series.est}
\begin{split}
& \Vert (-\Delta)^{\sigma/2 +1-\theta}(\lambda+\Delta)^{-1} \Vert_{\mathcal{V}^{1,\, 2}\rightarrow \mathcal{V}^{1,\, 2}}
\leq C_3|\lambda|^{\sigma/2 -\theta}, \\
& \Vert (-\Delta)^{\theta-\alpha/2} (\lambda +\Delta)^{-1}\Vert_{\mathcal{V}^{1,\, 2}\rightarrow \mathcal{V}^{1,\, 2}}
\leq C_3|\lambda|^{\theta - \alpha/2 -1}, \quad \forall \,\lambda \in S_\gamma.
\end{split}
\end{equation}

Moreover, from \eqref{Von-1} and \eqref{Von-2}, 
the last term in the right hand side of \eqref{A.expension} can be expanded as  
\begin{equation*}
\begin{split}
& (-\Delta)^{\sigma/2}(\lambda+\Delta)^{-1}\mathcal{B}\mathcal{R}(\lambda)(-\Delta)^{-\alpha/2}\\
& = (-\Delta)^{\sigma/2}(\lambda+\Delta)^{-1}\mathcal{B}\left \{1 -(\lambda +\Delta)^{-1}\mathcal{B}\right \}^{-1}
(\lambda +\Delta)^{-1} (-\Delta)^{-\alpha/2}\\
& = (-\Delta)^{\sigma/2}(\lambda+\Delta)^{-1}\mathcal{B} 
\sum_{j=0}^\infty \left \{(\lambda +\Delta)^{-1}\mathcal{B}\right\}^{j} (\lambda +\Delta)^{-1} (-\Delta)^{-\alpha/2}\\
& = (-\Delta)^{\sigma/2 +1-\theta}(\lambda+\Delta)^{-1} \mathcal{C}
\sum_{j=0}^\infty \left \{(\lambda +\Delta)^{-1}(-\Delta)\mathcal{C}\right\}^{j}(-\Delta)^{\theta-\alpha/2} (\lambda +\Delta)^{-1}. 
\end{split}
\end{equation*}

The estimates \eqref{C-ope}--\eqref{series.est} together with this expansion imply 
that the operator $(-\Delta)^{\sigma/2}(1+\Delta)^{-1}\mathcal{B}\mathcal{R}(\lambda)(-\Delta)^{-\alpha/2}$ is 
in $\mathcal{L}(\mathcal{V}^{1,\, 2}, \mathcal{V}^{1,\, 2})$. 
Moreover, 
\[\left \Vert (-\Delta)^{\sigma/2}(\lambda+\Delta)^{-1}\mathcal{B}\mathcal{R}(\lambda)(-\Delta)^{-\alpha/2} \right \Vert
_{\mathcal{V}^{1,\, 2}\rightarrow \mathcal{V}^{1,\, 2}} \leq C_3^2\, C_1^{-1}|\lambda|^{(\sigma-\alpha)/2-1}. \]

This completes the proof of the lemma.
\end{proof}

\begin{remark} Note that $\epsilon_1(\alpha,\sigma)$ may depend also on $\gamma$ but we shall ignore this dependence as $\gamma$
 is fixed throughout the 
paper. 
\end{remark}

Let us now define the Dunford integral
\begin{equation} \label{flow.def}
 e^{-\mathcal{A}t} =\frac{1}{2\pi i}\int_\Gamma e^{-\lambda t}(\lambda -\mathcal{A})^{-1} d\lambda, \quad t >0,
\end{equation}
where $\Gamma$ is a smooth curve in $S_\gamma$ which is oriented counterclockwise and 
connects $e^{-i\vartheta} \infty$ to $e^{i\vartheta} \infty$ for some $0< \gamma <\vartheta 
< \pi/2$.  Note that $D(\mathcal{A})$ may not be dense in $\mathcal{V}^{1,\, 2}_s$. 
However, from \eqref{R.s.est} and a simple extension of the standard theory of analytic semigroups 
(see, e.g., \cite[Proposition 1.1]{S}), we see that the integral in \eqref{flow.def} 
is well-defined as an operator from $\mathcal{V}^{1,\, 2}_{s}$ to $\mathcal{V}^{1,\, 2}_{s}$ and independent of
the choice of $\Gamma$. Moreover, $e^{-\mathcal{A} t}$ is a semigroup, and
\[ \frac{d}{dt} e^{-\mathcal{A}t } = -\mathcal{A} e^{-\mathcal{A}t}, \quad 
\Vert e^{-\mathcal{A}t} \Vert_{\mathcal{V}^{1,\, 2}_s \rightarrow \mathcal{V}^{1,\, 2}_s} \les 1, \quad \forall\,
t >0.  \]

Also, note that as a simple extension of  the standard semigroup theory, the 
property $\lim_{t\rightarrow 0^+} e^{-\mathcal{A}t} w = w$ 
 holds only for $w \in \overline{D(\mathcal{A})}$ which is not the same as $\mathcal{V}^{1,\, 2}_s$ (see 
\cite[Proposition 1.2]{S}). 

Our next goal is to apply Lemma \ref{Res} to extend $e^{-\mathcal{A}t}$ to a bounded operator from $\mathcal{V}^{1,\, 2}_{\alpha}$ 
into 
$\mathcal{V}^{1,\, 2}_{\sigma}$ for $0 \leq \sigma -\alpha \leq 2$ and $\alpha, \sigma \in (-2, 1)$. 
For such $\alpha$ and $\sigma$, recall that  $\epsilon_1(\alpha, \sigma)$ has been defined in Lemma \ref{Res}.

\begin{proposition}\label{A.decay} Let $\alpha, \sigma \in (-2, 1)$ be such that $|\sigma -\alpha| \leq 2$.
Assume that $\norm{U}_{\mathcal{V}^{1,\, 2}} < \epsilon_1(\alpha, \sigma)$. Then there exists a constant 
$C = C(\alpha, \sigma)$ such that for all $t >0$,
$$\left \Vert e^{-\mathcal{A}t}\right\Vert_{\mathcal{V}^{1,\, 2}_{\alpha} 
\rightarrow \mathcal{V}^{1,\, 2}_{\sigma}}
\leq C t^{(\alpha-\sigma)/2}, \quad \text{if~} \alpha \leq \sigma,$$
and 
$$\left \Vert e^{-\mathcal{A}t} -1 \right\Vert_{\mathcal{V}^{1,\, 2}_{\alpha} 
\rightarrow \mathcal{V}^{1,\, 2}_{\sigma}}  \leq Ct^{(\alpha-\sigma)/2}, \quad 
\text{if~}  \alpha \geq \sigma.
$$

\end{proposition}
\begin{proof} The proof of this proposition follows from a standard argument in the theory of semigroups (see \cite{Lu, Pa}). However, 
we present it here for the sake of completeness. Let $\gamma$ be fixed as in Lemma \ref{Res}. For each $t >0$, 
and $\gamma < \vartheta < \pi/2$, we let 
$\Gamma = \Gamma_1 \cup \Gamma_{2} \cup \Gamma_3$ with
$$\Gamma_1  =\{re^{-i\vartheta} : t^{-1} \leq r < \infty \}, \quad 
$$
$$\Gamma_2 =\{t^{-1}e^{-i\varphi} : \vartheta \leq \varphi \leq 2\pi -\vartheta\},  
$$
and
$$\Gamma_3 =\{re^{i\vartheta} : t^{-1} \leq r < \infty \}.   
$$

Then, it follows that from \eqref{R.s.est} and \eqref{flow.def} that
\[e^{-\mathcal{A}t} = \frac{1}{2\pi i}
\sum_{k=1}^3 \int_{\Gamma_k} e^{-\lambda t}(\lambda -\mathcal{A})^{-1} d\lambda. 
\]

Note that at this point, the above identity is only understood as an identity in
$\mathcal{L}(\mathcal{V}^{1,\, 2}_s, \mathcal{V}^{1,\, 2}_s)$ with $s \in (0,1)$. However, 
from Lemma \ref{Res}, we obtain
\begin{equation*}
\begin{split}
& \norm{ (-\Delta)^{\sigma/2}
\int_{\Gamma_1} e^{-\lambda t}(\lambda -\mathcal{A})^{-1} d\lambda (-\Delta)^{-\alpha/2}
}_{\mathcal{V}^{1,\, 2} \rightarrow \mathcal{V}^{1,\, 2}} \\
& =  \norm{ (-\Delta)^{\sigma/2}
\int_{\Gamma_1} e^{-\lambda t}(\lambda -\mathcal{A})^{-1} d\lambda (-\Delta)^{-\alpha/2}
}_{\mathcal{V}^{1,\, 2} \rightarrow \mathcal{V}^{1,\, 2}} \\
& \les \int_{t^{-1}}^\infty e^{- tr\cos(\vartheta)}r^{\frac{\sigma -\alpha}{2} -1} dr \\
&\les [t \cos(\vartheta)]^{\frac{\alpha-\sigma}{2}}\int_{\cos(\vartheta)}^\infty 
e^{-s} s^{\frac{\sigma -\alpha}{2} -1} ds
\les t^{\frac{\alpha-\sigma}{2}}.
\end{split}
\end{equation*}

Similarly, we also have the same estimate for the integral on $\Gamma_3$. Finally, using Lemma \ref{Res} again, 
we get
\begin{equation*}
\begin{split}
& \norm{ (-\Delta)^{\sigma/2}
\int_{\Gamma_2} e^{-\lambda t}(\lambda -\mathcal{A})^{-1} d\lambda (-\Delta)^{-\alpha/2}
}_{\mathcal{V}^{1,\, 2} \rightarrow \mathcal{V}^{1,\, 2}} \\
& \les t^{\frac{\alpha-\sigma}{2}} \int_{-\vartheta}^{\vartheta} e^{\cos(\varphi)} d\varphi
\les t^{\frac{\alpha-\sigma}{2}}. 
\end{split}
\end{equation*}

Thus, the first inequality in the lemma follows. 
The proof of the second one is similar. To see that, we use \eqref{flow.def} and \eqref{A.expension} to write
\begin{equation*}
\begin{split}
e^{-\mathcal{A}t} - 1 & = \int_{\Gamma} e^{-\lambda t}(\lambda +\Delta)^{-1} d\lambda -1 +
\int_{\Gamma} e^{-\lambda t} (\lambda +\Delta)^{-1}\mathcal{B} \mathcal{R}(\lambda) d\lambda \\
& = e^{\Delta t} -1 +  \int_{\Gamma} e^{-\lambda t} (\lambda +\Delta)^{-1}\mathcal{B} \mathcal{R}(\lambda) d\lambda. 
 \end{split}
\end{equation*}

The estimate of the first term on the right-hand side of the above equality follows from Theorem \ref{singular}.
The second one can be controlled exactly as what we just did using \eqref{se.est}. This completes the proof of 
the proposition.
\end{proof}
Next, note that if $||U||_{\mathcal{V}^{1,\, 2}} <\epsilon_1(s,s-2)$, 
and $f \in \mathcal{V}^{1,\, 2}_s$ 
it follows from Lemma \ref{Adomain} and Proposition \ref{A.decay} that 
$t^{-1}(e^{-\mathcal{A}t}f -f)$ and $\mathcal{A}f$ are both 
in $\mathcal{V}^{1,\, 2}_{s-2}$ and 
moreover,
\[ \norm{\frac{e^{-t\mathcal{A}}f-f}{t} + \mathcal{A}f}
_{\mathcal{V}^{1,\, 2}_{s-2}} \leq C(s)\norm{f}_{\mathcal{V}^{1,\, 2}_{s}}, 
\quad \forall~t >0. \]

However, this gives us no information 
on the differentiability of $e^{-\mathcal{A}t}$ at $t=0$. 
Our next result proves that $e^{-\mathcal{A}t}$ 
is differentiable at $0$ in a slightly different space.
\begin{proposition} \label{differentiability} 
Let $s, \sigma$ be two real numbers such that $0 < \sigma +2 \leq s \leq 
\sigma +4$, and $s \in (0,1)$.
Assume that 
\[ \norm{U}_{\mathcal{V}^{1,\, 2}} < \min\{\epsilon_1(s, s-2), \epsilon_1(s-2, \sigma+2)\}. \]
Then, for all 
$f \in \mathcal{V}^{1,\, 2}_s$, we have $ t^{-1}(e^{-\mathcal{A}t}f -f) +\mathcal{A}f$ 
is in $\mathcal{V}^{1,\, 2}_\sigma$, and moreover,
\begin{equation*}
 \left \Vert \frac{e^{-\mathcal{A}t}f -f}{t}  + \mathcal{A}f \right \Vert_{\mathcal{V}^{1,\, 2}_\sigma}
\leq C(s, \sigma) t^{\frac{s-\sigma}{2}-1}\norm{f}_{\mathcal{V}^{1,2}_s}, \quad \forall\,  t >0.
\end{equation*}
\end{proposition} 
\begin{proof} For each fixed $t>0$, from the Dunford integral \eqref{flow.def} 
and the change of variables $\mu=t\lambda$, we obtain
\begin{equation*} 
\begin{split}
\frac{e^{-t\mathcal{A}}f -f}{t} 
& = \frac{1}{2\pi i} \int_{\Gamma} \frac{e^{-\lambda t}}{t}\{(\lambda -\mathcal{A})^{-1} -\lambda^{-1} 
\}f d\lambda \\
& = \frac{1}{2\pi i}\int_{\Gamma'} 
\frac{e^{-\mu}}{\mu^2}\left (1- \frac{t}{\mu}\mathcal{A}\right )^{-1} \mathcal{A} f d\mu,
\end{split}
\end{equation*}
where $\Gamma' = \{t\lambda: \lambda \in \Gamma\}$ with $\Gamma$ being as in \eqref{flow.def}. 
Moreover, since
\[ \mathcal{A}f = -\frac{1}{2\pi i}\int_{\Gamma'} \frac{e^{-\mu}}{\mu^2}\mathcal{A}f d\mu, \]
it follows that
\[ 
\frac{e^{-t\mathcal{A}}f -f}{t} + \mathcal{A}f = 
\frac{1}{2\pi i} 
\int_{\Gamma'} \frac{e^{-\mu}}{\mu^2} \frac{t\mathcal{A}}{\mu} 
\left(1 -\frac{t}{\mu}\mathcal{A}\right)^{-1} 
\mathcal{A}f d\mu.\]

Since $||U||_{\mathcal{V}^{1,\, 2}} <\epsilon_1(s-2, \sigma+2)$, using 
Lemma \ref{Adomain} and Proposition \ref{Res}, we get
\begin{equation*}
\begin{split}
& \norm{\mathcal{A} \left(1 -\frac{t}{\mu}\mathcal{A}\right)^{-1} 
\mathcal{A}f}_{\mathcal{V}^{1,2}_{\sigma}} \\ 
& \leq \norm{\mathcal{A}}_{\mathcal{V}^{1,\, 2}_{\sigma +2} \rightarrow \mathcal{V}^{1,\, 2}_{\sigma}} \cdot 
\norm{\left(1 -\frac{t}{\mu}\mathcal{A}\right)^{-1}}_{\mathcal{V}^{1,\, 2}_{s-2} \rightarrow \mathcal{V}^{1,\, 2}_{\sigma+2}} \cdot
\norm{\mathcal{A}f}_{\mathcal{V}^{1,\, 2}_{s-2}} \\
& \les t^{\frac{s-\sigma}{2}-2}|\mu|^{\frac{\sigma-s}{2}+2}||f||_{\mathcal{V}^{1,\, 2}_s}.
\end{split}
\end{equation*}

Therefore, it follows that
\[\norm{\frac{e^{-t\mathcal{A}}f -f}{t} + \mathcal{A}f}_{\mathcal{V}^{1,\, 2}_\sigma}  
\leq C(s,\sigma)t^{\frac{s-\sigma}{2}-1}||f||_{\mathcal{V}^{1,\, 2}_s}, \]
and thus the proof is then complete.
\end{proof}

\begin{remark} It follows from the standard theory of semigroups 
\textup{(}see \cite{Lu, Pa, S}\textup{)} that  
\[ \lim_{t\rightarrow 0^+} \left [ \frac{e^{-\mathcal{A}t}f -f}{t} +\mathcal{A}f
\right ] =0 ,\quad \text{in the topology of }\, \mathcal{V}^{1,\, 2}_s \]
holds if and only if $f \in D(\mathcal{A}) \subset \mathcal{V}^{1,\, 2}_s$ and 
$\mathcal{A}f \in \overline{D(\mathcal{A})} \subset \mathcal{V}^{1,\, 2}_s$. However, 
this result is not sufficient for our purposes here.
\end{remark} 

\section{Stability of Stationary Solutions}\label{sec5}

Recall that $\delta_0 >0$ is defined in Theorem \ref{Stationary-Sol}. The main goal of this section 
is to prove Theorem \ref{stability-th}. To that end, we first determine the number $\delta_1$ claimed in Theorem \ref{stability-th}.
In order to do so we now fix $ 0 < \sigma_1 < 1/2$. Then by Theorem \ref{Stationary-Sol}, 
we can find 
$\delta_1 \leq \delta_0$ sufficiently small so that
for every $F$ with $\norm{F}_{\mathcal{V}^{1,\, 2}_{-2}} <\delta_1$, the solution $U$ of \eqref{NV} given by
Theorem \ref{Stationary-Sol} enjoys the estimate
\begin{equation*}
\begin{split}
\norm{U}_{\mathcal{V}^{1,\, 2}} < \min \{& 
\epsilon_1(\sigma_1-2, \sigma_1-2), \ \epsilon_1(\sigma_1,\sigma_1-2),\ 
\epsilon_1(\sigma_0, \sigma_0 -2),\\
& \epsilon_1(\sigma_0, \sigma_0),\ 
\epsilon_1(\sigma_0-2, \sigma_0),\ \epsilon_1(\sigma_1 -2, \sigma_1),\\ 
&\epsilon_1(1/2, -3/2),\ \epsilon_1(-3/2, \sigma_1)\}. 
\end{split}
\end{equation*}
The following result strengthens Proposition \ref{A.decay} in sense that it is uniform with respect to 
$\norm{U}_{\mathcal{V}^{1,\, 2}}$.
\begin{proposition}\label{A.uni}  Let $\sigma_0\in (1/2, 1)$. If $\norm{F}_{\mathcal{V}^{1,\, 2}_{-2}} < \delta_1$, then 
for every $\alpha, \sigma \in [\sigma_1 -2, \sigma_0]$ with $|\alpha - \sigma| \leq 2$, 
one has 
\begin{equation*}
\begin{split}
 \left \Vert e^{-\mathcal{A}t}\right\Vert_{\mathcal{V}^{1,\, 2}_{\alpha} 
\rightarrow \mathcal{V}^{1,\, 2}_{\sigma}}
 & \leq C(\sigma_0, \sigma_1) t^{(\alpha-\sigma)/2}, \quad \text{if~}  \alpha \leq \sigma, \\
 \left \Vert e^{-\mathcal{A}t} -1 \right\Vert_{\mathcal{V}^{1,\, 2}_{\alpha} 
\rightarrow \mathcal{V}^{1,\, 2}_{\sigma}} & \leq C(\sigma_0, \sigma_1) t^{(\alpha-\sigma)/2}, \quad 
\text{if~}  \alpha \geq \sigma.
\end{split}
\end{equation*}
\end{proposition}
\begin{proof} Since
\begin{equation*}
\begin{split}
\norm{U}_{\mathcal{V}^{1,\, 2}} < \min \{& 
\epsilon_1(\sigma_1-2, \sigma_1-2), \ \epsilon_1(\sigma_1,\sigma_1-2),\ 
\epsilon_1(\sigma_0, \sigma_0 -2),\\
& \epsilon_1(\sigma_0, \sigma_0),\ 
\epsilon_1(\sigma_0-2, \sigma_0),\ \epsilon_1(\sigma_1 -2, \sigma_1)\}, 
\end{split}
\end{equation*}
our proposition follows directly from Proposition \ref{A.decay} 
and the following interpolation result.
\end{proof}

\begin{proposition}\label{interpolation} For $s_0, s_1 <1$, and $0 < \theta <1$, 
let $s = (1-\theta)s_0 + \theta s_1$. Then, the space 
$\mathcal{V}^{1,\, 2}_s$ coincides with the complex interpolation space $[\mathcal{V}^{1,\, 2}_{s_0}, 
\mathcal{V}^{1,\, 2}_{s_1} ]_{\theta}$.
\end{proposition}

\begin{proof} For a definition of complex interpolation spaces we refer to the book \cite{BL}.
We first observe that for any $g\in\mathcal{S}'/\mathcal{P}$ and $\sigma\in\RR$ one has
\begin{equation}\label{LPweight}
\int_{\RR^n}|(-\Delta)^{\frac{\sigma}{2}}g(x)|^2w(x)dx \simeq \int_{\RR^n} \norm{\{\varPsi_{j}*g(\cdot)\}}_{\ell^{\sigma}_2}^2 w dx, 
\end{equation}
which holds for all weights $w$ belonging to the class $A_2$. For this see, e.g., 
Theorem 2.8, Remark 2.9, and Remark 4.5 in \cite{Bui}. In 
\eqref{LPweight}, the functions $\varPsi_j$, $j\in\mathbb{Z}$, are given by
$$\hat{\varPsi}_{j}(x)= \varPsi(2^{-j}x),$$
where $\varPsi $ is a function in $C^{\infty}_{0}(\RR^n)$ such
 that ${\rm supp} (\varPsi)=\{1/2 \leq |x|\leq 2\}$, $\varPsi(x)>0$ for $1/2<|x|<2$,
 and $\sum_{j=-\infty}^{\infty}\varPsi(2^{-j}x)=1$ for $x\not=0$.  Also,
for each $r\in\RR$ we have let $\ell^{r}_{2}$ denote the space of all sequences $\{a_j\}_{j=-\infty}^{\infty}$, $a_j\in \mathbb{C}$, 
such that
\begin{center}
$\norm{\{a_j\}}_{\ell^{r}_{2}}=\left\{\sum_{j=-\infty}^{\infty} (2^{jr}|a_j|)^2\right\}^{1/2}<+\infty.$
\end{center}

The equivalence \eqref{LPweight} and \cite[Lemma 3.1]{MV} then yield
\begin{equation}\label{LPcap}
\norm{g}_{\mathcal{V}^{1,\, 2}_{\sigma}} \simeq \norm{\norm{\{\varPsi_{j}*g(\cdot)\}}_{\ell^{\sigma}_2}}_{\mathcal{V}^{1,\,2}}, 
\qquad \sigma\in\RR.
\end{equation}

On the other hand, it can be seen that for  $s=(1-\theta)s_0+\theta s_1$ one has 
\begin{equation}\label{V.ell.inter}
 [\mathcal{V}^{1,\,2}(\ell^{s_0}_2), \mathcal{V}^{1,\,2}(\ell^{s_1}_2)]_{\theta}= \mathcal{V}^{1,\,2}(\ell^{s}_2)
\end{equation}
with equal norms.

Therefore, the proposition follows in a standard way using  \eqref{LPcap} and \eqref{V.ell.inter}.
\end{proof}

We shall prove Theorem \ref{stability-th} by using the next two Propositions. The first one asserts 
the existence and uniqueness of the solution to the integral equation \eqref{inte.form}. The second one
confirms that this solution is in fact the solution of the equation \eqref{KZ}. 

\begin{proposition}\label{stability-lemma} Let $\sigma_0 \in (1/2,1)$ be as in Theorem \ref{stability-th}.
There exists a sufficiently small positive number $\epsilon_0$ 
such that, for every $F \in \mathcal{V}^{1,\, 2}_{-2}$, $w^0 \in \mathcal{V}^{1,\, 2}$ satisfying 
$\norm{F}_{\mathcal{V}^{1,\, 2}_{-2}} <\delta_1$, $\nabla \cdot w^0 =0$ and $\norm{w^0}_{\mathcal{V}^{1,\, 2}} < \epsilon_0$, 
there is a unique, time-global solution $w(x,t)$ of the integral equation \eqref{inte.form} 
which satisfies
\[ \sup_{t >0}t^{1/4}\norm{w}_{\mathcal{V}^{1,\, 2}_{1/2}} \leq C\norm{w^0}_{\mathcal{V}^{1,\, 2}}, \]
and for every $\alpha \in [-1, 0]$, the estimate
\begin{equation} \label{al.Y}
\sup_{t >0 }t^{\alpha/2}\norm{w -w^0}_{\mathcal{V}^{1,\, 2}_\alpha} \leq C\norm{w^0}_{\mathcal{V}^{1,\, 2}}
\end{equation}
holds true. Moreover, for every $\sigma \in [0, \sigma_0]$, the solution $w$ also enjoys the estimate
\begin{equation} \label{sigma-w}
\sup_{t>0} t^{\sigma/2}\norm{w}_{\mathcal{V}^{1,\, 2}_\sigma} \leq C(\sigma_0) \norm{w^0}_{\mathcal{V}^{1,\, 2}}.
\end{equation}
\end{proposition}
\begin{proof} Let us define another space
\[ \mathbb{Y} = \{ f : (0,\infty) \rightarrow 
\mathcal{V}^{1,\, 2}_{1/2} \, \, \text{with}\, \, \norm{f}_{\mathbb{Y}} 
= \sup_{t > 0}t^{1/4} \norm{f(\cdot, t)}_{\mathcal{V}^{1,\, 2}_{1/2}} < \infty  \}.
\]

We are aiming to show, using Lemma \ref{fixedpoint},  that there exists a solution $w$ of \eqref{KZ} in $\mathbb{Y}$. To this end, 
let us set $y_0 = e^{-\mathcal{A} t}w^0$. Then, from Proposition \ref{A.uni}, it follows that
\begin{equation} \label{y0.est} \norm{y_0}_{\mathbb{Y}} = t^{1/4}\left \Vert y_0 \right\Vert_{\mathcal{V}^{1,\, 2}_{1/2}} 
\les \norm{w^0}_{\mathcal{V}^{1,\, 2}}.  \end{equation}
Thus, $y_0 \in \mathbb{Y}$. Now, we only need to estimate the following bi-linear map:
\[\wt{B}[w,v](t) = -\int_0^t e^{-\mathcal{A}(t-s)}\mathbb{P}\nabla\cdot[w(\cdot, s)\otimes v(\cdot, s)]ds, \quad 
w , v \in \mathbb{Y}.  \]

Indeed, it follows from Proposition \ref{A.uni}, Theorem \ref{riesz}, and Corollary \ref{pro} that
\begin{equation*}
\begin{split}
\left \Vert \wt{B}[w,v](t)\right \Vert_{\mathcal{V}^{1,\, 2}_{1/2}} & \leq
\int_0^t \left \Vert e^{-\mathcal{A}(t-s)}\mathbb{P}\nabla\cdot[w(\cdot, s)\otimes v(\cdot, s)]
\right \Vert_{\mathcal{V}^{1,\, 2}_{1/2}}ds\\
& \les \int_0^t (t-s)^{-3/4} \norm{\mathbb{P}\nabla\cdot[w(\cdot, s)\otimes v(\cdot, s)]}_{\mathcal{V}^{1,\, 2}_{-1}}\, ds \\
&  \les \int_0^t (t-s)^{-3/4} \norm{w(\cdot, s)\otimes v(\cdot, s)}_{\mathcal{V}^{1,\, 2}} \, ds.
\end{split}
\end{equation*}
Thus, by using H\"{o}lder's inequality and Theorem \ref{Sobcap}, we get
\begin{eqnarray*}
\left \Vert \wt{B}[w,v](t)\right \Vert_{\mathcal{V}^{1,\, 2}_{1/2}} &\les&  
\int_0^t (t-s)^{-3/4} \norm{|w(\cdot, s)|^2}_{\mathcal{V}^{1,\, 2}}^{1/2}
\norm{|v(\cdot, s)|^2}_{\mathcal{V}^{1,\, 2}}^{1/2}\,
 ds\\
&\les& \int_0^t (t-s)^{-3/4} \norm{w(\cdot, s)}_{\mathcal{V}^{1,\, 2}_{1/2}} \norm{v(\cdot, s)}_{\mathcal{V}^{1,\, 2}_{1/2}}\, ds \\
& \les&  \norm{w}_{\mathbb{Y}}\norm{v}_{\mathbb{Y}} \int_0^t (t-s)^{-3/4} s^{-1/2}\, ds \\
& \les & t^{-1/4}\norm{w}_{\mathbb{Y}} \norm{v}_{\mathbb{Y}}.
\end{eqnarray*}

This yields
\begin{equation} \label{Bwt.est} \norm{\wt{B}[w,v]}_{\mathbb{Y}} \les \norm{w}_{\mathbb{Y}}\norm{v}_{\mathbb{Y}}, 
\quad \forall \,  w, v \in \mathbb{Y}.
\end{equation}

It follows from Lemma \ref{fixedpoint} and the estimates \eqref{y0.est}--\eqref{Bwt.est}
 that there exists an $\epsilon_0 >0$ sufficiently small such that if $
\norm{w^0}_{\mathcal{V}^{1,\, 2}} <\epsilon_0$, there is a unique solution $w$ of 
\eqref{inte.form} such that
\[ \norm{w}_{\mathbb{Y}} \leq 2\norm{y_0}_{\mathbb{Y}} \les \norm{w^0}_{\mathcal{V}^{1,\, 2}}.  \]

Next, we shall prove \eqref{al.Y}. For every $\alpha \in [-1, \sigma_0]$, by Proposition \ref{A.uni} and 
as in the proof of the estimate of $\wt{B}$ on $\mathbb{Y} \times \mathbb{Y}$, we obtain
\begin{equation} \label{w-eAw0}
 \begin{split}
\Vert w -e^{-\mathcal{A}t}w^0 \Vert_{\mathcal{V}^{1,\, 2}_\alpha} & = \norm{\wt{B}[w,w](t)}_{\mathcal{V}^{1,\, 2}_\alpha}  \\
& \les \int_0^t (t-s)^{-\frac{\alpha+1}{2}} \norm{w(\cdot, s) \otimes w(\cdot, s)}_{\mathcal{V}^{1,\, 2}}\, ds\\
&\les \int_0^t (t-s)^{-\frac{\alpha+1}{2}} \norm{w(\cdot, s)}_{\mathcal{V}^{1,\, 2}_{1/2}}^2\, ds\\
& \les \norm{w}_{\mathbb{Y}}^2\int_0^t (t-s)^{-\frac{\alpha+1}{2}}s^{-1/2} ds \les t^{-\alpha/2}\norm{w}_{\mathbb{Y}}^2.
 \end{split}
\end{equation}

Moreover, if we restrict $\alpha \in [-1, 0]$, Proposition \ref{A.uni} also yields
\begin{equation} \label{eA-1}
\norm{(e^{-\mathcal{A}t} -1)w^0}_{\mathcal{V}^{1,\, 2}_\alpha} 
\les t^{-\alpha/2}\norm{w^0}_{\mathcal{V}^{1,\, 2}}.
\end{equation}

Thus, it follows from \eqref{w-eAw0} and \eqref{eA-1} that for  $\alpha \in [-1,0]$ we have
\begin{eqnarray*}
t^{\alpha/2}\norm{w(\cdot, t)-w^0}_{\mathcal{V}^{1,\, 2}_\alpha} &=& 
t^{\alpha/2}\norm{w(\cdot, t)-e^{-\mathcal{A}t}w^0 +e^{-\mathcal{A}t}w^0 -w^0}_{\mathcal{V}^{1,\, 2}_\alpha}\\
&\les& \norm{w}_{\mathbb{Y}}^2 + \norm{w^0}_{\mathcal{V}^{1,\, 2}}  \les \norm{w^0}_{\mathcal{V}^{1,\, 2}},
\end{eqnarray*}
which proves \eqref{al.Y}. 

Finally, for $\sigma \in [0, \sigma_0]$, Proposition \ref{A.uni} implies that
\[ \norm{e^{-\mathcal{A}t} w^0}_{\mathcal{V}^{1,\, 2}_\sigma} \les t^{-\sigma/2}\norm{w^0}_{\mathcal{V}^{1,\, 2}}. \]
Using this and \eqref{w-eAw0} (with  $\sigma$ in place of $\alpha$), we get \eqref{sigma-w}. 
This completes the proof of the lemma.
\end{proof}
To prove that the solution $w(x,t)$ of the integral equation \eqref{inte.form} 
obtained in Proposition \ref{stability-lemma} is the solution of \eqref{KZ}, we 
need to following  inequality:
\begin{lemma}\label{Ho-the} Let $\sigma$ be in $(0,1)$ and $s_1, s_2$ be in $[0,1)$ such that
$\sigma + s_1 + s_2 =1$. Then, there is $C= C(\sigma, s_1, s_2) >0$ such that
\[\norm{f\otimes g}_{\mathcal{V}^{1,\, 2}_{-\sigma}} \leq C\norm{f}_{\mathcal{V}^{1,\, 2}_{s_1}}
\norm{g}_{\mathcal{V}^{1,\, 2}_{s_2}}. \] 
\end{lemma}
\begin{proof} Using Theorem  \ref{Sobcap} with $\alpha = \sigma, 
p =\frac{2}{1+\sigma}$ and then applying H\"{o}lder inequality, we get
\[\norm{f\otimes g}_{\mathcal{V}^{1,\, 2}_{-\sigma}} 
\les \sup_{K}\left \{ \frac{\int_K|f|^{\frac{2}{1-s_1}}}{{\rm cap}_{1,\, 2}(K)} \right \}^{\frac{1-s_1}{2}} \cdot
\sup_{K}\left \{ \frac{\int_K |g|^{\frac{2}{1-s_2}}}{{\rm cap}_{1,\, 2}(K)} \right \}^{\frac{1-s_2}{2}}. \] 
Here, the suprema are taken over all compact sets $K$ with ${\rm cap}_{1,\, 2}(K) >0$. 
Our desired result follows by again applying the Theorem \ref{Sobcap}. 
\end{proof}
\begin{proposition}\label{KZ.solution} For every $F, w^0$ satisfying $\norm{F}_{\mathcal{V}^{1,\, 2}_{-2}} < \delta_1$, 
$\nabla \cdot w^0 =0$ and 
$\norm{w^0}_{\mathcal{V}^{1,\, 2}} < \epsilon_0$, let $w(x,t)$ be a solution of the integral 
equation \eqref{inte.form}
as in Proposition \ref{stability-lemma}. Then, $w$ satisfies \eqref{KZ} in the sense of 
tempered distributions. 
\end{proposition}
\begin{proof} For each $0 < a < b <\infty$ and for $a \leq \tau < t \leq b$, by \eqref{inte.form} we have
\begin{align} \label{in.tau}
w(\cdot, t) - w(\cdot, \tau) & = [e^{-(t-\tau)\mathcal{A}}-1]w(\cdot, \tau) \\ \notag
& \quad - \int_\tau^t e^{-(t-s)\mathcal{A}}\mathbb{P}\nabla \cdot [w(\cdot, s)\otimes w(\cdot, s)]ds.
\end{align}

Then, using Theorem \ref{singular}, Theorem \ref{Sobcap} and Proposition \ref{A.uni}, we obtain
\begin{align}\label{w.holder}
& \norm{w(\cdot, t) - w(\cdot, \tau)}_{\mathcal{V}^{1,\, 2}_{\sigma_1}} \\ \notag
& \les (t-\tau)^{(1/2-\sigma_1)/2}\norm{w(\tau)}_{\mathcal{V}^{1,\, 2}_{1/2}} + 
\int_{\tau}^t(t-s)^{-\frac{1+\sigma}{2}}\norm{w(s)}_{\mathcal{V}^{1,\, 2}_{1/2}}^2ds
\\ \notag 
& \leq C(a,b) \left [ (t-\tau)^{1/4 -\sigma_1/2} + (t-\tau)^{(1-\sigma_1)/2} \right ]\norm{w^0}_{\mathcal{V}^{1,\, 2}}. 
\end{align}

On the other hand, it follows from Theorem \ref{singular} and Lemma \ref{Ho-the} that
\begin{equation}\label{nonl-hol}
\begin{split}
& \norm{\mathbb{P}\nabla [w(\cdot, t)\otimes w(\cdot, t) -
 w(\cdot, \tau)\otimes w(\cdot, \tau) ]}_{\mathcal{V}^{1,\, 2}_{\sigma_1 -2}} \\
& \les \norm{w(\cdot, t)\otimes w(\cdot, t) -
 w(\cdot, \tau)\otimes w(\cdot, \tau)}_{\mathcal{V}^{1,\, 2}_{\sigma_1-1}}\\
& \les \norm{w(\cdot, t) - w(\cdot, \tau)}_{\mathcal{V}^{1,\, 2}_{\sigma_1}} \,
\sup_{\tau\in [a, b]}\norm{w(\cdot, \tau)}_{\mathcal{V}^{1,\, 2}}\\
& \les \norm{w(\cdot, t) - w(\cdot, \tau)}_{\mathcal{V}^{1,\, 2}_{\sigma_1}}. 
\end{split}
\end{equation}

Now, for each fixed $t >0$, let $a < b < \infty$ be two numbers such that $a < t <b$. Then, 
for each $t_1, t_2$ such that $a < t_1 < t < t_2 < b$, from \eqref{in.tau}, we get
\[ \frac{w(t_2) - w(t_1)}{t_2 -t_1} =\frac{e^{-(t_2-t_1)\mathcal{A}} -1}{t_2 -t_1}w(t) 
-\mathbb{P}\nabla \cdot[w(t)\otimes w(t)] + T_1 - T_2 - T_3. \]
Here,
\begin{equation*}
\begin{split}
T_1 & = \frac{e^{-(t_2-t_1)\mathcal{A}} -1}{t_2-t_1}[w(t_1) -w(t)], \\
T_2 & = \frac{1}{t_2-t_1}\int_{t_1}^{t_2}e^{-(t_2 -s)\mathcal{A}}\mathbb{P}\nabla\cdot 
[w(s)\otimes w(s) -
w(t)\otimes w(t)] ds, \\
T_3 & = \frac{1}{t_2-t_1}\int_{t_1}^{t_2}\left\{e^{-(t_2-s)\mathcal{A}} -1\right\}
\mathbb{P}\nabla\cdot [w(t)\otimes w(t)] ds.
\end{split}
\end{equation*}

Since
\[ \norm{U}_{\mathcal{V}^{1,\, 2}} < \min\{\epsilon_1(1/2, -3/2), \, \epsilon_1(-3/2, \sigma_1)\}, \]
we can apply Proposition \ref{differentiability} with $s =1/2$ and $\sigma = \sigma_1-2$ to get
\[ \lim_{t_2,\, t_1 \rightarrow t}\frac{e^{-(t_2-t_1)\mathcal{A}} -1}{t_2 -t_1}w(t)  
= -\mathcal{A} w(t), \quad \text{in} \quad \mathcal{V}^{1,\, 2}_{\sigma_1 -2}.  \]

On the other hand, by applying Proposition \ref{A.uni}, Theorem \ref{Sobcap} and 
the estimates \eqref{w.holder}, \eqref{nonl-hol}, we find
$$\norm{T_1}_{\mathcal{V}^{1,\, 2}_{\sigma_1-2}}  \les \norm{w(t_1) -w(t)}_{\mathcal{V}^{1,\, 2}_{\sigma_1}} 
\rightarrow 0, \quad \text{as} \quad t_1, t_2 \rightarrow t,
$$
and 
\begin{align*}
 \norm{T_2}_{\mathcal{V}^{1,\, 2}_{\sigma_1-2}} & \les \frac{1}{t_2-t_1}\int_{t_1}^{t_2}
\norm{\mathbb{P}\nabla\cdot [w(s)\otimes w(s) - w(t)\otimes w(t)]}_{\mathcal{V}^{1,\, 2}_{\sigma_1-2}} ds\\
& \les \sup_{t_1 \leq s \leq t_2}\norm{w(s) -w(t)}_{\mathcal{V}^{1,\, 2}_{\sigma_1}} 
\rightarrow 0, \quad \text{as} \quad t_1, t_2 \rightarrow t.
\end{align*}

Moreover, it follows from Theorem \ref{singular}, Proposition \ref{A.uni} and Lemma \ref{Ho-the} 
that 
\begin{align*}
\norm{T_3}_{\mathcal{V}^{1,\, 2}_{\sigma_1-2}}  & \les  \frac{1}{t_2-t_1}\int_{t_1}^{t_2}(t_2-s)^{1/4}
\norm{\mathbb{P}\nabla\cdot [w(t)\otimes w(t)]}_{\mathcal{V}^{1,\, 2}_{\sigma_1-\frac{3}{2}}} ds \\
& \les (t_2-t_1)^{1/4}\norm{w(t)\otimes w(t)}_{\mathcal{V}^{1,\, 2}_{\sigma_1-\frac{1}{2}}} \\
& \les (t_2-t_1)^{1/4}\norm{w(t)}_{\mathcal{V}^{1,\, 2}_{\frac{1}{4} + \frac{\sigma_1}{2}}}^2 
\rightarrow 0, \quad \text{as} \quad t_1, t_2 \rightarrow t.
\end{align*}

In conclusion, we have
\[w_t + \mathcal{A} w(t) + \mathbb{P}\nabla\cdot [w(t)\otimes w(t)] =0, \quad
\text{in} \quad \mathcal{S}', \]
and this completes the proof of the proposition.
\end{proof}

We are finally in a position to prove Theorem \ref{stability-th}.
\begin{proof}[Proof of Theorem \ref{stability-th}] Let $\delta_1$ and $\epsilon_0$ 
be as in Proposition \ref{stability-lemma}
and let $u(t)=w(t)-U$ where $w(t)$ is the unique solution of 
the integral equation \eqref{inte.form} with initial datum
$w^0= u_0-U$ as obtained in Proposition \ref{stability-lemma}. 
It follows from Proposition \ref{stability-lemma} and Proposition \ref{KZ.solution} that
$u(t)$ is a solution of \eqref{NSE} which satisfies all of the estimates stated in 
Theorem \ref{stability-th}. The uniqueness of $u$ 
follows directly from the fact that if $u$ 
is any solution of \eqref{NSE} satisfying the estimate \eqref{V1/2}, 
then $w =u-U$ is a solution of \eqref{inte.form} with initial datum
$w^0 =u_0-U$ and $w$ satisfies all the estimates in Proposition \ref{stability-lemma}.
\end{proof}
\begin{remark}\label{ref2} \label{MC-remark} We would like to point out that one 
can follow the approach in \cite{CC, CCP, MC} \textup{(}see also \cite[Chapter 15]{L-R}\textup{)} to  prove 
the existence of $u$ satisfying only 
\eqref{sigma-u.est} with $\sigma =0$ in a much simpler way. Indeed, instead of \eqref{inte.form}
 we write
\[ w(t) = e^{t\Delta }w^0 -\int_0^t e^{(t-s)\Delta}\mathbb{P}\nabla \cdot[ w \otimes U + U \otimes w + w \otimes w] ds, \]
and let $Y = \{w \in (L^\infty(\mathcal{V}^{1,2}))^n: \sup_{t>0} |w(t,\cdot)| \in \mathcal{V}^{1,2}\}$ 
with norm 
\[ \norm{w}_{Y} = \norm{\sup_{t>0 }|w(t, \cdot)|}_{\mathcal{V}^{1,2}} . \]

Then by Theorem \ref{IV} we have $\sup_{t>0}|e^{t\Delta }w^0|\leq C {\rm\bf M} w^0\in \mathcal{V}^{1,\, 2}$. On the other hand,
for $(v, w)\in Y\times Y$, writing $V= \sup_{t>0}|v(t, \cdot)|$ and $W= \sup_{t>0}|w(t, \cdot)|$ we find 
\begin{eqnarray*}
\left | \int_0^t e^{(t-s)\Delta}\mathbb{P}\nabla \cdot (v \otimes w) ds\right | &\leq& C \int_0^t\int_{\RR^n} \frac{ V(y)  W(y)}{(\sqrt{t-s}+|x-y|)^{n+1}} ds dy\\
&\leq& C {\rm\bf I}_1 *(VW).
\end{eqnarray*}

This gives 
\begin{equation*}
\norm{ \int_0^t e^{(t-s)\Delta}\mathbb{P}\nabla \cdot (U \otimes w+ w \otimes U ) ds}_{Y} \leq C 
\norm{U}_{\mathcal{V}^{1,\, 2}}\norm{w}_{Y}, 
\end{equation*}
and 
\begin{equation*}
\norm{ \int_0^t e^{(t-s)\Delta}\mathbb{P}\nabla \cdot ( v\otimes w) ds}_{Y} \leq C 
\norm{v}_{Y}\norm{w}_{Y}. 
\end{equation*}

Now, the existence of $u$ can be proved by the standard fixed point argument, which also implies 
\eqref{sigma-u.est} when $\sigma =0$. 
\end{remark}
\ \\ 
{\bf Acknowledgement.} The authors would like to thank the anonymous referee for valuable comments and suggestion that help improve the manuscript of the paper.  In particular, the authors would like to thank the referee for pointing 
out Remark \ref{ref2}.



\begin{thebibliography}{xxxxxx}
\bibitem[Ad]{Ad} D. R. Adams, {\it A note on Riesz potentials}, Duke Math. J. {\bf 42} (1975), 99--105.

\bibitem[AP]{AP} D. R. Adams  and   M. Pierre, {\it Capacitary strong type estimates in
semilinear problems}, Ann. Inst. Fourier (Grenoble), {\bf 41} (1991), 117--135.


\bibitem[BL]{BL} J. Bergh and J. L\"ofstr\"om, {\it Interpolation Spaces}. 
Springer--Verlag, Berlin--Heibelberg--New York, 1976.

\bibitem[BBIS]{BBIS} C.  Bjorland, L. Brandolese, D. Iftimie, and M. E. Schonbek, {\it $L^p$-solutions of the 
steady-state Navier-Stokes equations with rough external forces},  
Comm. Partial Differential Equations  {\bf 36}  (2011),  216--246.

\bibitem[BP]{BP} J. Bourgain and N. Pavlovi\'c, {\it Ill-posedness of the Navier-Stokes equations in a critical space in 3D}, 
J. Funct. Anal. {\bf 255} (2008), 2233--2247.


\bibitem[Bo]{Bo} G. Bourdaud,  {\it R\'ealisations des espaces de Besov homog\`enes},  Ark. Mat. {\bf 26} (1988),  41--54.

\bibitem[Bui]{Bui} Bui Huy Qui,  {\it Weighted Besov and Triebel spaces: Interpolation by the real method},  
Hiroshima Math. J. {\bf 12} (1982),  581--605.

\bibitem[CCP]{CCP} C. P. Calder{\'o}n, {\it Existence of weak solutions for the Navier-Stokes equations 
with initial data in $L^p$}, Trans. Amer. Math. Soc. {\bf 318} (1990), 179--200.

\bibitem[CC]{CC} C. P. Calder{\'o}n, {\it Addendum to the paper: 
``Existence of weak solutions for the Navier-Stokes equations with initial data in $L^p$"}, 
Trans. Amer. Math. Soc. {\bf 318} (1990),  201--207.

\bibitem[MC]{MC} M. Cannone, {\it 
Ondelettes, paraproduits et Navier-Stokes}, Diderot Editeur, Paris, 1995.

\bibitem[ChWW]{ChWW} S.-Y.~A. Chang, J.~M. Wilson, and T.~H. Wolff,
{\it Some weighted norm inequalities concerning the Schr\"odinger operators},
Comment. Math. Helv. {\bf 60} (1985), 217--246.

\bibitem[Fef]{Fef} C. Fefferman, {\it The uncertainty principle}, Bull. Amer. Math. Soc. {\bf 9} (1983), 129--206.

\bibitem[HMV]{HMV}  K. Hansson, V. G. Maz'ya, and I. E. Verbitsky, {\it Criteria of solvability for 
multidimensional Riccati equations},  Ark. Mat. {\bf 37} (1999), 87--120.

\bibitem[PG]{PG}  P. Germain, {\it Multipliers, paramultipliers, and weak-strong uniqueness 
for the Navier-Stokes equations}. J. Differential Equations {\bf 226} (2006),  373--428.



\bibitem[Ka1]{Ka1} T. Kato, {\it Strong $L^{p}$-solutions of the Navier-Stokes equation in ${\mathbb R}^{m}$, 
with applications to weak solutions}, Math. Z. {\bf 187} (1984),  471--480. 

\bibitem[Ka2]{Ka2} T. Kato, {\it Strong solutions of the Navier-Stokes equation in Morrey spaces}, 
Bol. Soc. Brasil. Mat. (N.S.) {\bf 22} (1992),  127--155.

\bibitem[KV]{KV} N. J. Kalton and I. E. Verbitsky, {\it Nonlinear equations and weighted norm inequalities},
Trans. Amer. Math. Soc. {\bf 351} (1999), 3441--3497.

\bibitem[KY1]{KY1} H. Kozono and M. Yamazaki, {\it Semilinear heat equations and the Navier-Stokes 
equation with distributions in new function spaces as initial data},
Comm. Partial Differential Equations {\bf 19} (1994), 959--1014.

\bibitem[KY2]{KY2} H. Kozono and M. Yamazaki, {\it The stability of small stationary solutions in 
Morrey spaces of the Navier-Stokes equation}, 
Indiana U. Math. J. {\bf 44} (1995), 1307--1335.

\bibitem[KT]{KT} H. Koch and D. Tataru, {\it Well-posedness for the Navier-Stokes equations},
Advances in Math. {\bf 157} (2001) 22--35.

\bibitem[KW]{KW} D. S.  Kurtz and R. L. Wheeden, {\it Results on weighted norm inequalities for multipliers}, 
Trans. Amer. Math. Soc. {\bf 255} (1979) 343--362.

\bibitem[La]{La} N. S. Landkof, {\it Foundations of Modern Potential Theory},  Springer-Verlag, Berlin--Heibelberg--New York, 1972.

\bibitem[L-R]{L-R} P. G. Lemari\'e-Rieusset, {\it Recent Developments in the Navier-Stokes Problem},
Chapman \& Hall/CRC Press, Boca Raton, 2002.

\bibitem[LR]{LR} P. G. Lemari\'e-Rieusset, {\it 
The Navier-Stokes equations in the critical Morrey-Campanato space},
Rev. Mat. Iberoam. {\bf 23} (2007),  897--930. 

\bibitem[{L-RM}]{L-RM} P. G. Lemari\'e-Rieusset and R. May, {\it 
Uniqueness for the Navier-Stokes equations and multipliers between Sobolev spaces}. 
Nonlinear Anal. {\bf 66} (2007),  819--838.

\bibitem[Lu]{Lu} A. Lunardi, {\it Analytic Semigroups and Optimal Regularity in Parabolic Problems},
Progress in Nonlinear Differential Equations and their Applications, 16. Birkh\"{a}user Verlag, Basel, 1995.

\bibitem[MS1]{MS1} V. G. Maz'ya,  and T. O. Shaposhnikova, {\it The Theory of Multipliers in Spaces of
Differentiable Functions}, Pitman, New York, 1985.

\bibitem[MS2]{MS2} V. G. Maz'ya and T. O. Shaposhnikova, {\it Theory of Sobolev multipliers. With applications to 
differential and integral operators}, Grundlehren der Mathematischen Wissenschaften  {\bf 337}. Springer-Verlag, 
Berlin, 2009. xiv+609 pp.

\bibitem[MV]{MV} V. G. Maz'ya and E. I. Verbitsky, {\it Capacitary inequalities for fractional integrals, 
with applications to partial differential equations and Sobolev multipliers},  Ark. Mat.  {\bf 33}  (1995),  81--115.

\bibitem[Me]{Me} Y. Meyer, {\it Wavelets, Paraproducts and Navier-Stokes Equations}, Current developments in mathematics, 
1996, Internat. Press,
Cambridge, MA 02238--2872 (1999).

\bibitem[M-S]{M-S} S. Montgomery-Smith, {\it Finite time blow up for a Navier-Stokes like equation}, Proc. Amer. Math. Soc. {\bf 129} (2001) 3025--3029.

\bibitem[Pa]{Pa} A. Pazy, {\it Semigroups of Linear Operators and Applications to Partial Differential Equations}, 
Applied Mathematical Sciences, 44. Springer-Verlag, New York, 1983.

\bibitem[P]{P} C. P\'erez, {\it Two weighted inequalities for potential and fractional type maximal operators},
Indiana Univ. Math. J.  {\bf 43}  (1994),  663--683.

\bibitem[Ph]{Ph} N. C. Phuc, {\it Quasilinear Riccati type equations with super-critical exponents}, 
Comm. Partial Differential Equations {\bf 35} (2010), 1958--1981.

\bibitem[PhV]{PhV} N. C. Phuc and I. E. Verbitsky, {\it Quasilinear and Hessian equations of Lane-Emden type}, 
Ann. Math. {\bf 168} (2008), 859--914. 

\bibitem[S]{S}  E. Sinestrari, {\it On the abstract Cauchy problem of parabolic type in spaces of continuous functions},
J. Math. Anal. Appl.  {\bf 107} (1985),   16--66.

\bibitem[Tay]{Tay} M. E. Taylor, {\it Analysis of Morrey spaces and applications to Navier-Stokes and other evolution equations}, 
Comm. Partial Differential Equations {\bf 17} (1992), 1407--1456.

\bibitem[VW]{VW} I. E. Verbitsky and R. L. Wheeden, {\it Weighted norm inequalities for integral operators},
Trans. Amer. Math. Soc. {\bf 350} (1998), 3371--3391.


\bibitem[Yo]{Yo} T. Yoneda, {\it Ill-posedness of the 3D-Navier-Stokes equations in a generalized Besov 
space near $\textup{BMO}^{-1}$}, J. Funct. Anal. {\bf 258} (2010),  3376--3387.

\end{thebibliography}
\end{document}